\documentclass{amsart}
\usepackage{amsthm}
\usepackage{amsmath}
\usepackage{amsfonts}
\usepackage{amssymb}
\usepackage{a4wide}
\usepackage{epsfig}
\usepackage{comment}

\usepackage[foot]{amsaddr}

\newtheorem{theorem}{\textbf{Theorem}}[section]
\newtheorem{lemma}[theorem]{\textbf{Lemma}}

\newtheorem{proposition}[theorem]{\textbf{Proposition}}

\newcommand{\e}{\epsilon}
\newcommand{\ds}{\displaystyle}

\theoremstyle{remark}
\newtheorem{remark}[theorem]{\textbf{Remark}}

\newcommand{\cacher}[1]{}

\def\cT{\mathcal{T}}
\def\ds{\displaystyle}

\def\cC{\mathcal{C}}

\def\cE{\mathcal{E}}
\def\cA{\mathcal{A}}

\def\D{\mathrm{diam}}
\def\cG{\mathcal{G}}

\def\vy{\mathbf{y}}
\def\vF{\mathbf{F}}
\def\vtau{\mathbf{\tau}}
\def\Jac{\mathrm{Jac}}
\def\yhv{\vy_h}

\def\wchi{\widetilde{\chi}}

\def\wh{\widehat}

\def\qedclaim{\hfill$\triangle$\smallskip}

\def\uy{\overline{y}}
\def\utau{\overline{\tau}}

\def\dmax{d_{\mathrm{max}}}

\title{On the diameter of random planar graphs}

\author{Guillaume Chapuy$^\ast$ \and \'Eric Fusy$^\star$ \and Omer
Gim\'enez$^\dagger$ \and Marc Noy$^\ddagger$}
\address{ \rm $^\ast$CNRS, LIAFA, UMR 7089, Universit\'e Paris Diderot - Paris 7,
Case 7014, 75205 Paris Cedex 13, France.}
\address{ \rm
$^\star$CNRS, LIX, UMR 7161, \'Ecole Polytechnique, 91128 Palaiseau Cedex, France.}
\address{ \rm
$^\dagger$Dept. de Llenguatges i Sistemes Inform\`atics, Universitat Polit\`ecnica de Catalunya,
Barcelona, Spain.}
\address{ \rm 
$^\ddagger$Dept. de Matem\`atica Aplicada II, Universitat Polit\`ecnica de
Catalunya, Barcelona, Spain. }
\address{ \rm G.C. and \'E.F.  partially supported by the European Research Council (grant ExploreMaps -- ERC StG 208471) and by the French Agence Nationale de la Recherche (grant Cartaplus -- ANR~12-JS02-001-01).}
\address{ \rm O.G. and M.N. partially supported by grants MTM2011-24097 and
DGR2009-SGR1040.}

\begin{document}

\begin{abstract}
We show that the diameter $\D(G_n)$ of a random  labelled  connected planar graph with $n$
vertices is equal to $n^{1/4+o(1)}$, in probability. More precisely, there exists
a constant $c>0$ such that
$$P(\D(G_n)\in(n^{1/4-\e},n^{1/4+\e}))\geq 1-\exp(-n^{c\e})$$ for $\e$ small enough and
$n\geq n_0(\e)$. We prove similar statements for 2-connected and 3-connected planar graphs and maps.
\end{abstract}

\maketitle

\section{Introduction}

A map is a connected planar graph with a given embedding in the plane.
The diameter of random maps  has attracted a lot of attention since the pioneering work by
Chassaing and Schaeffer~\cite{ChSc04} on the radius $r(Q_n)$ of
random quadrangulations with $n$ vertices, where they show that
$r(Q_n)$ rescaled by~$n^{1/4}$ converges as $n\to\infty$ to an
explicit continuous distribution related to the Brownian snake~\cite{LegSnake}.
This convergence was shown to hold for large families of planar
maps~\cite{MaMi,Mie06}, and it was conjectured that random maps of size $n$ rescaled by
$n^{1/4}$ converge in some sense to a continuum object, the \emph{Brownian
map}~\cite{MaMo,Leg}.
In recent years, several properties of the limiting object 
have been obtained
\cite{LegPa08, Mie08},
and the convergence result was proved very recently
independently by Miermont and Le Gall~\cite{Miermont:convergence,
LeGall:convergence}.
At the combinatorial level, the two-point
function of quadrangulations has surprisingly a simple exact expression,
a beautiful result found in~\cite{BoDFGu03}  that allows one to derive easily the
limit  distribution, rescaled by~$n^{1/4}$, of the distance between two randomly chosen vertices in a random quadrangulation.
In contrast, little is known about the profile of random \emph{unembedded}
connected planar graphs, even if it is strongly believed that the results should be similar as in the embedded case.
As a general remark, readers familiar with random graphs should observe that random planar graphs are in general more difficult to study than Erd\H os-R\'{e}nyi models, since the edges are \emph{not} drawn independently. 

Our main result in this paper is a large deviation statement for the diameter,
which strongly supports the belief that $n^{1/4}$ is the
right scaling order. We say that a property $A$, defined for all
values $n$ of a parameter, holds  asymptotically almost surely, a.a.s.\ for short, if
$$
    P(A) \to 1, \qquad \hbox{ as } n\to \infty.
   $$
 In this paper we need a certain rate of convergence of the probabilities. Suppose property $A $ depends on a real number $\epsilon >0$, usually very small. Then we say that $A$ holds a.a.s.\ with exponential rate if
there is a constant $c>0$, such that for every $\epsilon$ small enough there exists an integer $n_0(\epsilon)$ so that
\begin{equation}\label{eq:exp_small}
    P(\hbox{not } A) \le e^{-n^{c\epsilon}}    \qquad   \hbox{ for all  } n \ge n_0(\epsilon).
     \end{equation}

\medskip
The diameter of a graph (or map) $G$ is denoted by $\D(G)$.
The main results proved in this paper are the following.

\begin{theorem}\label{theo:main}
The diameter of a random connected labelled  planar graph with $n$  vertices is in the interval
$(n^{1/4- \e}, n^{1/4+\e})$ a.a.s.\ with exponential rate.
\end{theorem}

\begin{theorem}\label{theo:main2}
Let $1<\mu<3$.
The diameter of a random connected labelled planar graph with $n$
vertices and $\lfloor \mu n\rfloor$ edges
is  in the interval $(n^{1/4-\e},n^{1/4+\e})$ a.a.s.\ with exponential
rate.
\end{theorem}

These are the first results obtained on the diameter of random planar graphs. They give the right order of magnitude and show the connection to the well-studied problem of the radius of random quadrangulations. It is still open and seems technically very involved to show a limit distribution for the profile or radius of a random connected planar graph rescaled by
$n^{1/4}$.
Other extremal parameters that have been analyzed recently in random planar graphs using analytic techniques are the size of the largest $k$-connected component \cite{GNR,PS} and the maximum vertex degree \cite{DGN3,DGNPS}.

The results for planar graphs contrast with the so-called ``subcritical'' graph families,
such as trees, outerplanar graphs, and series-parallel graphs,
where the diameter is in the interval $(n^{1/2-\e},n^{1/2+\e})$
a.a.s.\ with exponential rate; see Section~\ref{sec:subcritical} at the end of the article.

Let us give a brief sketch of the proof.
Recall that a graph is $k$-connected if one needs to delete at least $k$ vertices
to disconnect it ($2$-connected graphs are assumed to be loopless,
 $3$-connected graphs are assumed to be loopless and simple).
First we prove the result for planar maps
via quadrangulations, using a bijection with labelled trees by Schaeffer that
keeps track of a distance parameter.
Then we prove the result for
2-connected maps using the fact that a random map has a large
2-connected core with non-negligible probability.
A similar argument allows us to extend the result to 3-connected
maps, which proves it also for 3-connected planar graphs, since by Whitney's theorem
they have a unique embedding in the sphere. We then reverse the
previous arguments and go first to 2-connected and then  to connected
planar graphs, but this  is not straightforward. One difficulty is
that the largest 3-connected component of a random 2-connected planar graph
does not have the typical ratio between number of edges and number
of vertices, and this is why we must study maps with a given weight
at vertices, so as to adjust the
ratio between edges and vertices. In addition, we must show that there is
a 3-connected component of size $n^{1-\e}$ a.a.s.\ with
exponential rate, and similarly for 2-connected components. Finally, we must show
that the height of the tree associated to the decomposition of a
2-connected planar graph into 3-connected components is at most $n^\e$,
and similarly for the tree of the decomposition of a connected planar graph
into 2-connected components.

\section{Preliminaries}
In this section we recall first some easy inequalities given by generating
functions. Then we describe the chain of correspondences and decompositions
that will allow us to carry large deviation estimates for the diameter,
starting from quadrangulations (and labelled trees associated to them)
and all the way down to connected planar graphs. In the sequel,
the diameter of a graph $G$ (whether a tree, a planar graph or a map) is denoted $\D(G)$.

\subsection{Saddle bounds and exponentially small tails}
Let $f(z)=\sum_nf_nz^n$ be a series with nonnegative coefficients
and let $x>0$ be a value such that $f(x)$ converges; in particular
$x$ is at most the radius of convergence $\rho$. Then we have the
following elementary  inequality for $n\geq 0$:
\begin{equation}\label{eq:saddle}
f_n\leq f(x)x^{-n}.
\end{equation}
When minimized over $x$, this inequality is called \emph{saddle-point bound}.

A bivariate version yields a
lemma that will be used several times; it provides a simple criterion
to ensure that the distribution of a parameter has an exponentially fast decaying
tail. First let us give some terminology. A \emph{weighted combinatorial class}
is a class of combinatorial objects (such as graphs, trees or maps) $\cA=\cup_n\cA_n$  endowed with a weight-function
$w:\cA\mapsto\mathbb{R}_+$. We write $|\alpha|=n$ if $\alpha\in\cA_n$.
The \emph{weighted distribution} in size $n$
is the unique distribution on $\cA_n$ proportional to the weight: $P(\alpha)\propto w(\alpha)$ for every $\alpha\in\cA_n$.

\begin{lemma}\label{lem:tail}
Let $\cA=\cup_n\cA_n$ be a weighted combinatorial class, $\chi:\cA\to\mathbb{N}$ a
parameter on $\mathcal A$, and let
$A(z,u)=\sum_{\alpha\in\cA}w(\alpha)z^{|\alpha|}u^{\chi(\alpha)}$.
Let $\rho >0$ be the
dominant singularity of $A(z,1)$, and let $A_n=[z^n]A(z)$. Assume that, for some $\alpha>0$,
$$
  A_n = \Omega(n^{-\alpha}\rho^{-n}).
$$
Assume also that there exists $u_0 >1$ such that $A(\rho,u_0)$ converges.

Then $\chi(R_n)\leq n^{\e}$ a.a.s.\ with exponential rate (under the weighted
distribution).
\end{lemma}
\begin{proof}
We have $P(\chi(R_n) = k)=[z^nu^k]A(z,u)/[z^n]A(z,1)$. A bivariate
version of~\eqref{eq:saddle} ensures that $[z^nu^k]A(z,u)\leq
A(\rho,u_0)\rho^{-n}u_0^{-k}=O(\rho^{-n}e^{-ck})$, where
$c=\log(u_0)$. Hence $P(\chi(R_n) = k)=O(n^{\alpha}e^{-ck})$.
This directly implies that $\chi(R_n)\leq n^{\e}$ a.a.s.\ with exponential rate.
\end{proof}

\subsection{Maps}\label{sec:maps}
A \emph{planar map} (shortly
called a map here) is a connected unlabelled graph embedded in the
oriented sphere up to isotopic deformation. Loops and multiple edges are allowed.
 A \emph{rooted map} is a map where
an edge is marked and oriented. Rooting is enough to avoid symmetry issues
(this contrasts with unembedded planar graphs, where labelling vertices or edges
is necessary to avoid symmetries). 
The face to the left of the root is called the \emph{outer face};
this face is taken as the infinite face in plane representations (e.g. in Figure~\ref{fig:bij_quad_tree},
left part).
A \emph{quadrangulation} is a map where all faces have degree $4$. Notice that an isthmus contributes twice to the degree of a face.

\subsubsection{Labelled trees and quadrangulations}\label{sec:s_bij}
We recall Schaeffer's bijection (itself a reformulation of an earlier
bijection by Cori and Vauquelin~\cite{CoriVa})
between labelled trees and quadrangulations.
A \emph{rooted plane tree} is a rooted map with a unique face.
A \emph{labelled tree} is a rooted plane tree with an integer label $\ell_v\in\mathbb{Z}$ on each vertex $v$ so that the labels of the end-points of each edge $e=(v,v')$ satisfy $|\ell_v-\ell_{v'}|\leq 1$, and such that the root vertex has label $0$. The minimal (resp. maximal) label in the tree is denoted
$\ell_{min}$ (resp. $\ell_{max}$).
A \emph{bicolored} labelled tree is a labelled tree endowed with a 2-coloring of the vertices (in black and white) such that
vertices of odd labels are of one color and vertices of even labels are of the other color.
Such a tree is called \emph{black-rooted} (resp. \emph{white-rooted}) if the root-vertex is black
(resp. white).
A \emph{bicolored quadrangulation} is a quadrangulation endowed with a 2-coloring
of its vertices (in black and white) such that adjacent vertices have different colors.
Such a 2-coloring is unique once the color of a given vertex is specified.
A rooted quadrangulation will be assumed to be endowed with the unique 2-coloring
such that the root-vertex is black.


\begin{figure}
\begin{center}
\includegraphics[width=12cm]{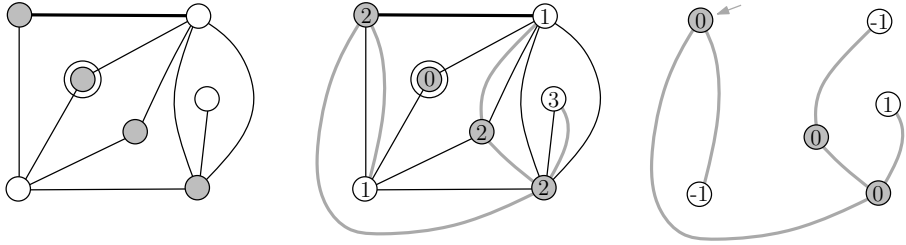}
 \end{center}
 \caption{Left: A bicolored quadrangulation with a marked vertex (surrounded) and a marked
edge (bolder).
Right: the associated bicolored labelled tree.}
 \label{fig:bij_quad_tree}
 \end{figure}

\begin{theorem}[Schaeffer~\cite{S-these}, Chapuy, Marcus, Schaeffer~\cite{ChMaSc09}]\label{theo:bij_quad_trees}
Bicolored quadrangulations with a marked vertex $v_0$ and a marked edge
are in bijection with bicolored labelled trees.
Each face of  a bicolored quadrangulation $Q$ corresponds to an
edge in the associated bicolored labelled
tree $\tau$.
Each non-marked vertex $v$ of $Q$ corresponds to a vertex $v$ of the same color in $\tau$, such that
 $\ell_v-\ell_{min}+1$ gives the distance from $v$ to $v_0$
in $Q$.
\end{theorem}

An example is shown in Figure~\ref{fig:bij_quad_tree}; see~\cite{ChMaSc09} for a detailed
description of the bijection.
Define the \emph{label-span} of $\tau$ as the quantity $L(\tau)=\ell_{max}(\tau)-\ell_{min}(\tau)$. It follows from the bijection in  Theorem~\ref{theo:bij_quad_trees} that  $L(\tau)+1$
is the radius of $Q$ centered at $v_0$. Hence
\begin{equation}\label{eq:diam_quad}
L(\tau)+1\leq \D(Q)\leq 2L(\tau)+2.
\end{equation}

\subsubsection{Quadrangulations and maps}\label{sec:bij_quad_maps}
We recall a classical bijection between rooted quadrangulations
with $n$ faces (and thus $n+2$ vertices)
and rooted maps with $n$ edges. Starting from $Q$
endowed with its canonical 2-coloring, add in each face a new edge
connecting the two diagonally opposed black vertices. Return the rooted map $M$
formed by the newly added edges and the black vertices, rooted at the
edge corresponding to the root-face of $Q$, and with same root-vertex as $Q$; see
Figure~\ref{fig:bij_quad_map}.
Conversely, to obtain $Q$ from $M$, add a new white
vertex $v_f$
inside each face $f$ of $M$ and add new edges from $v_f$
to every corner around $f$; then delete all edges from $M$, and take as root-edge
of $Q$
the one corresponding to the incidence root-vertex/outer-face in $M$.
\begin{figure}
\begin{center}
\includegraphics[width=12cm]{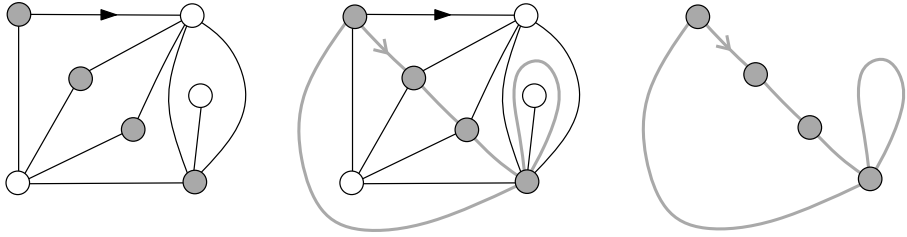}
 \end{center}
 \caption{Left: A rooted quadrangulation. Right: the associated rooted map.}
 \label{fig:bij_quad_map}
 \end{figure}
Clearly, under this bijection, vertices of a map correspond to black vertices of the associated quadrangulation, and faces correspond to white vertices.
Let $M$ be a
rooted map with $n$ edges and let $Q$ be the associated rooted
quadrangulation (with $n+2$ vertices). Every path $b_1b_2\dots b_k$
in $M$ yields a path $b_1w_1b_2\dots w_{k-1}b_k$ in $Q$, where
$w_i$ is the white vertex corresponding to the face to the left of
$(b_i,b_{i+1})$. Hence $\D(Q) \le 2\ \!\D(M)$.
Let $x=b_1w_1b_2w_2\dots b_k=y$ be a
path in $Q$, where the $b_i$ are black and the $w_i$ are
white. Let $f_i$ be the  face in $M$ corresponding to
$b_i$. Then we can find a path in $M$ between $x$ and $y$ of length at most
$k + \mathrm{deg}(f_1) + \cdots + \mathrm{deg}(f_k)$. Therefore, calling $\Delta(M)$ the maximal
face-degree in $M$, we obtain $\D(M)\leq \D(Q)\cdot\Delta(M)$. We thus
obtain the following inequalities that we use for estimating the diameter
of random maps from estimates of the diameter of random quadrangulations:
\begin{equation}\label{eq:diamMQ}
 \D(Q)/2\leq \D(M)\leq \D(Q)\cdot\Delta(M).
 \end{equation}

\subsubsection{The 2-connected core of a map}

\begin{figure}
\begin{center}
\includegraphics[width=12.5cm]{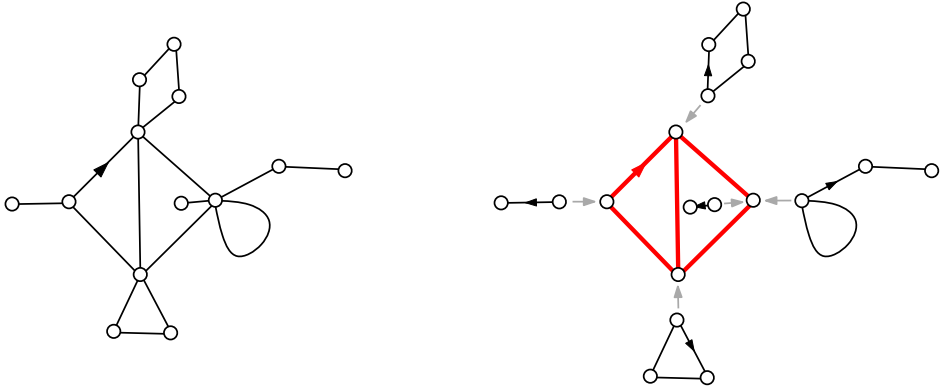}
 \end{center}
 \caption{A rooted map is obtained from a $2$-connected map (the core) where
at each corner a rooted map is possibly inserted.}
 \label{fig:2c-core}
 \end{figure}

It is convenient here to consider the map consisting of a single loop as 2-connected
(all 2-connected maps with at least two edges are loopless).
As described by Tutte in~\cite{Tu63}, a rooted map $M$ is obtained by taking
a rooted 2-connected map $C$, called the \emph{core} of $M$, and then
inserting at each corner $i$ of $C$ an arbitrary rooted map $M_i$;
see Figure~\ref{fig:2c-core}.
The maps $M_i$ are called the \emph{pieces} of $M$.
The following inequalities will be used to estimate the diameter
of  random rooted 2-connected maps from estimates of the  diameter of random rooted maps:
\begin{equation}
\D(C)\leq \D(M)\leq \D(C)+2\cdot\mathrm{max}_i(\D(M_i)).
\end{equation}
The first inequality is trivial, and the second one follows
from the fact that a diametral path in $M$ either stays in a single piece,
or it connects two different pieces while traversing edges of $C$.

\subsubsection{The 3-connected core of a 2-connected map}\label{sec:3ccore}

\begin{figure}
\begin{center}
\includegraphics[width=12.5cm]{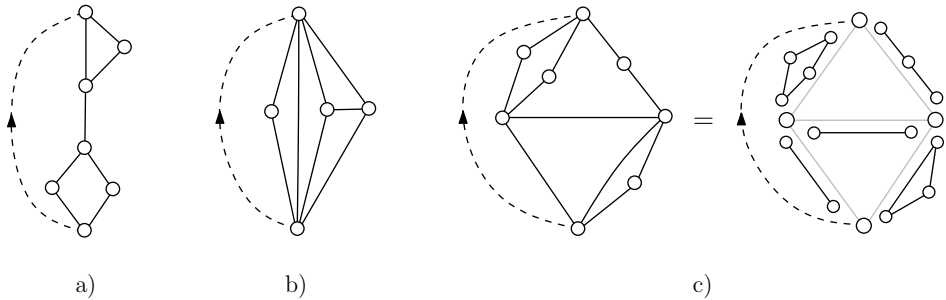}
 \end{center}
 \caption{(a) A network made of $3$ networks assembled in series. (b)  A network made of $3$ networks (one of which is an edge) assembled in parallel. (c) A network with a 3-connected core
(which is a $K_4$) where each
edge is substituted by a network.}
 \label{fig:3c-core}
 \end{figure}

A \emph{plane network} is a map $M$ with two marked vertices in the outer face,
called the \emph{poles} of $M$ ---the $0$-pole and the $\infty$-pole---
such that adding an edge $e$ between these two vertices yields a rooted
$2$-connected map, called the \emph{completed map} of the network.
Conversely a plane network is just obtained from a 2-connected map with at least two edges
by deleting the root-edge, the origin and end of the root-edge being distinguished respectively
as the $0$-pole and the $\infty$-pole.
 A \emph{polyhedral network} is a plane network such that the poles are not adjacent and such that
the completed map is $3$-connected.
As shown by Tutte~\cite{Tu63} (see Figure~\ref{fig:3c-core}), a plane network
$C$ is either a series or parallel composition of plane networks,
or it is obtained from a polyhedral network $T$ where each edge $e$
is possibly substituted by a plane network $C_e$, identifying the end-points
of $e$ with those of the root of $C_e$. In that case
 $T$ is called the \emph{3-connected core} of $C$ and the components $C_e$ are
 called the \emph{pieces} of $C$.
Calling $d_e$ the degree of the root face of $C_e$, we obtain
the following inequalities, which will be used to get a diameter estimate
for random 3-connected maps from a diameter estimate for random 2-connected maps:
 \begin{equation}\label{eq:lowerDC}
 \D(T)\leq \D(C)\leq \D(T)\cdot\mathrm{max}_{e\in T}(d_e)+2\ \!\mathrm{max}_{e\in T}(\D(C_e)).
 \end{equation}
The first inequality is trivial. The second one follows from the fact that a diametral
path $P$ in $C$ starts in a piece, ends in a piece, and in between it
passes by vertices $v_1,\ldots,v_k$ of $T$ such that for $1\leq i<k$,  $v_i$ and $v_{i+1}$ are adjacent in $T$ ---let $e=\{v_i,v_{i+1}\}$---
and $P$ travels in the piece $C_e$ to reach $v_{i+1}$ from $v_i$;
since $P$ is geodesic, its length in $C_e$ is bounded by the distance from $v_i$
to $v_{i+1}$, which is clearly bounded by $d_e$.

\subsection{Planar graphs}
By a theorem of Whitney, a  $3$-connected planar graph has a unique embedding on the oriented sphere. Hence $3$-connected planar maps are equivalent to $3$-connected planar graphs.
Once we have an estimate for the diameter of random $3$-connected maps, hence also for random $3$-connected planar graphs,
we can carry such an estimate up to random connected planar graphs, using a well known decomposition of a connected
planar graph into $3$-connected components, via a decomposition into $2$-connected components.
We now describe these decompositions and give inequalities relating the diameter of
a graph to the diameters of its components.

\subsubsection{Decomposing a connected planar graph into 2-connected components}\label{subsec:1to2}
There is a well-known decomposition of a graph into 2-connected components~\cite{Mohar,Tutte}.
Given a connected graph $C$, a \emph{block} of $C$ is a maximal 2-connected  subgraph of $C$.
The set of blocks of $C$ is denoted by $\mathfrak{B}(C)$.
A vertex $v\in C$ is said to be \emph{incident} to a block $B\in\mathfrak{B}(C)$ if $v$ belongs to $B$. The \emph{Bv-tree} is the
 bipartite graph $\tau(C)$ with vertex-set $V(C)\cup\mathfrak{B}(C)$,
and edge-set given by the incidences between the vertices and the blocks of $C$; see Figure~\ref{fig:blocks}. It is easy to see that  $\tau(C)$ is actually
a tree.

\begin{figure}\label{Bv-tree}
\begin{center}
\includegraphics[width=12.5cm]{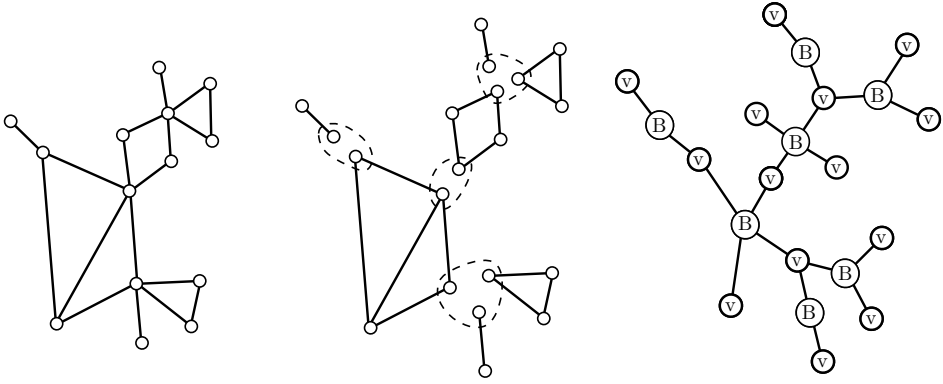}
 \end{center}
 \caption{Decomposition of a connected graph into blocks, and the associated Bv-tree.}
 \label{fig:blocks}
 \end{figure}

We will use the following inequalities to get a diameter estimate for random connected
planar graphs from a diameter estimate for random 2-connected planar graphs.
For a  connected planar graph $G$, with Bv-tree $\tau$ and blocks $B_1,\ldots,B_k$, we have:
\begin{equation}\label{eq:ineq_block}
\displaystyle
\mathrm{max}_i(\D(B_i))\leq\D(G)\leq \mathrm{max}_i(\D(B_i))\cdot\D(\tau).
\end{equation}
The first inequality is trivial. The second inequality follows from the fact that a diametral
path in $G$ induces a path $P$ in $\tau$ of length at most $\D(\tau)$, and the length
``used'' by each block $B$ along $P$ is at most $\D(B)$.

\subsubsection{Decomposing a 2-connected planar graph into 3-connected components}\label{subsec:2to3}
In this section we recall Tutte's decomposition of a 2-connected graph into
3-connected components~\cite{Tu63}.
First, we define connectivity modulo a pair of vertices. Let $G$ be a 2-connected graph
(possibly with multiple edges)
 and $\{u,v\}$ a pair of vertices of $G$. Then $G$ is said to be \emph{connected modulo $[u,v]$} if $u$ and $v$ are not adjacent
and if $G\backslash \{u,v\}$ is connected.

Define a \emph{2-separator} of a 2-connected graph $G$ as a partition of the edges of $G$, $E(G)=E_1\uplus E_2$ with $|E_1|\geq 2$ and $|E_2|\geq 2$, such that $E_1$ and $E_2$ can be separated by the removal of a pair of vertices $u,v$. 
A 2-separator $E_1,E_2$ is called a \emph{split-candidate}, 
denoted by $\{E_1,E_2,u,v\}$,  if $G[E_1]$ is connected modulo $[u,v]$ and $G[E_2]$ is 2-connected
(for $E'\subseteq E(G)$, we use the notation $G[E']$ to denote the subgraph of $G$ made of edges in $E'$ and vertices
incident to at least one edge from $E'$). 
 Figure~\ref{fig:split}(a) gives an example of a split-candidate, where $G[E_1]$ is connected modulo $[u,v]$ but not 2-connected, while $G[E_2]$ is 2-connected but not connected modulo $[u,v]$.

\begin{figure}\label{split}
\begin{center}
\includegraphics[width=9.2cm]{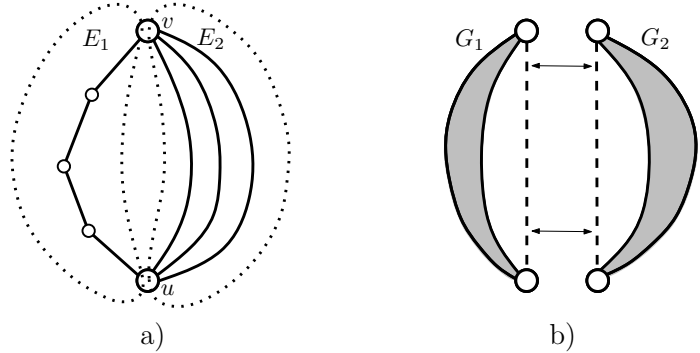}
 \end{center}
 \caption{(a) Example of a split-candidate. (b) Splitting a graph along a virtual edge.}
 \label{fig:split}
 \end{figure}

As described below, split-candidates make it possible to decompose completely a 2-connected graph into 3-connected components. We consider here only 2-connected graphs with at least three edges (graphs with less edges are degenerated for this decomposition). Given a split-candidate $S=\{E_1,E_2,u,v\}$ in a 2-connected graph $G$  (see Figure~\ref{fig:split}(b)),
the corresponding \emph{split operation} is defined as follows, see Figure~\ref{fig:split}(b):
\begin{itemize}
\item
an edge $e$, called a \emph{virtual edge}, is added between $u$ and $v$,
\item
the graph $G[E_1]$ is separated from the graph $G[E_2]$
by cutting along the edge $e$.
\end{itemize}
Such a split operation yields two graphs $G_1$ and $G_2$,  which correspond respectively to $G[E_1]$ and $G[E_2]$, together with $e$ as a real edge; see Figure~\ref{fig:split}(b). The graphs $G_1$ and $G_2$ are said to be \emph{matched} by the virtual edge $e$. It is easily checked that $G_1$ and $G_2$ are 2-connected (and have at least three edges). The splitting process can be repeated until no split-candidate remains.

\begin{figure}
\begin{center}
\includegraphics[width=13cm]{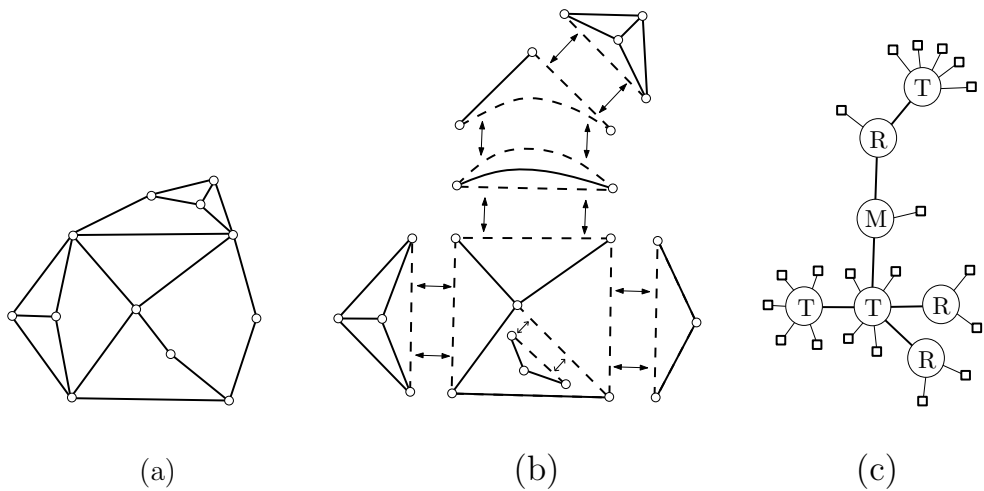}
 \end{center}
 \caption{(a) A 2-connected graph, (b) decomposed into bricks. (c) The associated RMT-tree.}
\label{fig:bricks}
 \end{figure}

As shown by Tutte in~\cite{Tutte}, the structure resulting from the split operations is independent of the order in which they are performed. It is a collection of graphs, called the \emph{bricks} of $G$, which
are articulated around virtual edges; see Figure~\ref{fig:bricks}(b).
By definition of the decomposition, each brick has no split-candidate;
Tutte shows that such graphs are either multiedge-graphs (M-bricks) or ring-graphs (R-bricks), or 3-connected graphs with at least four vertices (T-bricks).

The \emph{RMT-tree} of $G$ is the graph $\tau(G)$ whose inner nodes correspond to the bricks of $G$, and the edges between such vertices correspond to the virtual edges of $G$ (each virtual edge matches two bricks);
additionally the leaves  of $\tau(G)$ correspond to the real (not virtual) edges of $G$;
see Figure~\ref{fig:bricks}.
The graph $\tau(G)$ is indeed a tree~\cite{Tutte}.
By maximality of the decomposition, it is easily checked that $\tau(G)$ has no two adjacent R-bricks nor two adjacent M-bricks.

We will use the following inequalities to get a diameter estimate for random 2-connected
planar graphs from a diameter estimate for random 3-connected planar graphs (which
are equivalent to random $3$-connected maps, by Whitney's theorem). For
a 2-connected planar graph $G$, with RMT-tree $\tau$, bricks $B_1,\ldots,B_k$,
and $\mathcal{E}_{\mathrm{virt}}$ as the set of pairs of vertices of $G$ connected by a virtual edge, we have:
\begin{equation}\label{eq:diam3c}
\displaystyle
\mathrm{max}_i(\D(B_i))\leq\D(G)\leq \mathrm{max}_i(\D(B_i))\cdot(\D(\tau)+1)\cdot  \mathrm{max}_{(u,v)\in\mathcal{E}_{\mathrm{virt}}}\mathrm{Dist}_G(u,v).
\end{equation}
The first inequality is trivial. The second inequality follows from the following facts:
\begin{itemize}
\item
 a diametral
path $P_G$ in $G$ induces a path $P$ in $\tau$ (of length at most $\D(\tau)$),
\item
for each brick $B$  traversed by $P_G$ ($B$ corresponds to a vertex of $\tau$
that lies on $P$, there are $\D(\tau)+1$ such vertices), 
the path $P_G$ induces a path $P_B=(v_0,\ldots,v_k)$ in $B$,
where each edge $\{v_i,v_{i+1}\}$ is either a virtual edge or a real edge of $G$.
\item
 the length of $P_G$
``used'' when traversing an edge $e=\{v_i,v_{i+1}\}\in P_B$
is at most the distance between $v_i$ and $v_{i+1}$ in $G$.
\end{itemize}
Hence the length of $P_G$ ``used by $B$'' is at most $\D(B)\cdot \mathrm{max}_{(u,v)\in\mathcal{E}_{\mathrm{virt}}}\mathrm{Dist}(u,v)$,
so that the total length of $P_G$ is given by the second inequality.



\section{Diameter estimates for families of maps}
In this section we consider families of maps, starting with quadrangulations and ending with $3$-connected maps. In each case we show that for a random map $G$ of size $n$
 in such a family,
we have $\D(G)\in (n^{1/4-\e},n^{1/4+\e})$ a.a.s.\ with exponential rate, where 
the size parameter $n$ is typically the number
of edges or the number of faces.
In order to carry later on (in Section~\ref{sec:diam_graphs}) 
such estimates from 3-connected maps
to connected planar graphs, 
we need to show that such concentration properties hold more generally in a weighted setting.
More precisely, if a combinatorial class $\cG=\cup_n\cG_n$
(each $\gamma\in\cG$ has a size $|\gamma|\in\mathbb{N}$, and the set of
objects of $\cG$ of size $n$ is denoted $\cG_n$)
has an additional weight-function $w(\cdot)$, then the \emph{generating function}
of $\cG$ is
$$
G(z)=\sum_{\alpha\in\cG}w(\alpha)z^{|\alpha|},
$$
and the \emph{weighted} probability
distribution in size $n$ assigns to each map $G\in\cG_n$ the probability
$$
\mathbb{P}(G)=\frac{w(G)}{C_n},\ \  \mathrm{with}\ C_n=\sum_{G\in\cG_n}w(G).
$$
Typically, for planar maps and planar graphs, the weight will be of the form $w(G)=x^{\chi(G)}$, with $x$ a fixed positive real value and $\chi$ a parameter
such as the number of vertices; in that case  the  terminology will be
``a random map of size $n$ with weight $x$ at vertices''.

\subsection{Quadrangulations}
From Schaeffer's bijection in Section~\ref{sec:s_bij} it is easy to show large deviation results for the diameter
of a quadrangulation. The basic
idea, originating in~\cite{ChSc04}, is that the typical depth $k$ of a vertex in the tree is $n^{1/2}$,
and the typical discrepancy of the labels along a branch is $k^{1/2}=n^{1/4}$.
We use a fundamental result from~\cite{FlGa93}, namely that
under very general conditions the height of a random tree of size $n$
from a given family is in $(n^{1/2-\e},n^{1/2+\e})$ a.a.s.\ with exponential rate.

Let $y(z)=\sum_{\tau\in\cT} z^{|\tau|}w(\tau)$ be the weighted generating function of some combinatorial class $\cT$
(typically $\cT$ is a class of rooted trees),
and denote by $\rho$ the radius of convergence of $y(z)$, assumed to be strictly positive.
Assume $y\equiv y(z)$
 satisfies an equation of the form
\begin{equation}\label{eq:tree}
y=F(z,y),
\end{equation}
with $F(z,y)$ a bivariate function with nonnegative coefficients,
nonlinear in $y$, analytic around $(0,0)$, such that $F(0,0)=0$ and $F(0,y)=0$. By the
non-linearity of~\eqref{eq:tree} with respect to $y$, $y(\rho)$ is
finite; let $\tau=y(\rho)$. Equation~\eqref{eq:tree} is called
\emph{admissible} if $F(z,y)$ is analytic at $(\rho,\tau)$, in which case $F_y(\rho,\tau)=1$.  
Equation~\eqref{eq:tree} is called \emph{critical} if $F(z,y)$ is not analytic at $(\rho,\tau)$ but 
 $F(\rho,\tau)$ converges as a sum and $F(\rho,\tau)<1$, which is equivalent to the fact
that $y'(z)$ converges at $\rho$.  
 A \emph{height-parameter} for~\eqref{eq:tree} is a nonnegative integer
parameter $\xi$ for structures in $\cT$ such that $
y_h(z)=\sum_{\tau\in\cT,\xi(\tau)\leq h}w(\tau)z^{|\tau|}$
satisfies
$$
y_{h+1}(z)=F(z,y_h(z))\ \ \mathrm{for}\ h\geq 0,\ \ \ y_0=0.
$$

\begin{lemma}[Theorem 1.3. in~\cite{FlGa93}]\label{lem:height}
Let $\cT$ be a combinatorial class endowed with a weight-function $w(\cdot)$ so that the corresponding weighted generating function
$y(z)$ satisfies an equation of the form~\eqref{eq:tree}, and such
that~\eqref{eq:tree} is admissible.

Let $\xi$ be a height-parameter for~\eqref{eq:tree} and let $T_n$ be
taken at random in $\cT_n$ under the weighted distribution in size
$n$. Then $\xi(T_n)\in(n^{1/2-\e},n^{1/2+\e})$ a.a.s.\ with
exponential rate.
\end{lemma}

\begin{remark}
Theorem~1.3 in~\cite{FlGa93} actually gives bounds for the coefficients
$[z^n]y_h(z)$ from which Lemma~\ref{lem:height} directly follows, observing
that $P(\xi(T_n)>h)=([z^n](y(z)-y_h(z)))/[z^n]y(z)$ and $P(\xi(T_n)\leq h)=[z^n]y_h(z)/[z^n]y(z)$.  
The authors of~\cite{FlGa93} prove the result for plane trees, 
  then they claim that all the arguments in the proof hold for
any system of the form $y=z\phi(y)$. The arguments hold even more generally
for any admissible system
of the form $y=F(z,y)$.
\end{remark}

The next proposition is proved as a warm up, what we will need is a weighted version
that is more technical to prove.
\begin{proposition}\label{prop:diamQ}
The diameter of a random rooted quadrangulation with $n$ faces
is, a.a.s.\ with exponential rate, in the interval $ (n^{1/4- \e},
n^{1/4+\e})$.
\end{proposition}
\begin{proof}
When the number of black vertices is not taken into account, the statement of
Theorem~\ref{theo:bij_quad_trees} simplifies:
it gives a $1$-to-$2$ correspondence between labelled trees having $n$ edges and rooted  quadrangulations
having $n$ faces and a secondary marked vertex; 
once again for a vertex $v$ of a labelled tree $\tau$, the quantity $\ell_v-\ell_{\mathrm{min}}+1$
gives the distance of $v$ from the marked vertex in the associated quadrangulation.
According to~\eqref{eq:diam_quad}, we just have to show that, for a uniformly random labelled tree
$\tau$ with $n$ vertices, $L(\tau)=\ell_{\mathrm{max}}-\ell_{\mathrm{min}}$ is in
$(n^{1/4-\e},n^{1/4+\e})$ a.a.s.\ with exponential rate. Since the
label either increases by $1$, stays equal, or decreases by $1$
along each edge (going away from the root), the series $T(z)$ of
labelled trees counted according to vertices satisfies
$$
T(z)=\frac{z}{1-3T(z)},
$$
and the usual height of the tree is a height-parameter for this equation.
The equation is clearly admissible (the singularity is at $1/12$ and $T(1/12)=1/6$),
hence by Lemma~\ref{lem:height}
 the height is in $(n^{1/2-\e},n^{1/2+\e})$ a.a.s.\ with exponential rate.
So in a random labelled tree there is a.a.s.\ with exponential rate a path $B$
of length $k=n^{1/2-\e}$ starting from the root. The labels along $B$ form a random walk
with increments $+1$, $0$, $-1$, each with probability $1/3$. Classically the maximum
of such a walk is at least $k^{1/2-\e}$ (which is at least
$n^{1/4-\e}$) a.a.s.\ with exponential rate. Hence  the label
of the vertex $v$ on $B$ at which the maximum occurs is at least the label of the root-vertex plus
$n^{1/4-\e}$, so $\ell_{max}\geq n^{1/4-\e}$ a.a.s.\ with exponential rate.
Since $\ell_{min}\leq 0$, this proves the lower
bound.

For the upper bound (already proved in~\cite{ChSc04}), since the height is at most
$n^{1/2+\e}$ a.a.s.\ with exponential rate, the same is true for the
depth $k$ of a random vertex $v$ in a random labelled tree of
size $n$. The labels along the path from the root to $v$ form a random
walk of length $k$, the maximum of which is at most $k^{1/2+\e}$ a.a.s.\ with exponential rate.
Hence
$|\ell(v)|\leq n^{(1/2+\e)^2}$ a.a.s.\ with exponential rate,
so the same holds for the property $|\ell(v)|\leq
n^{1/4+\e}$. Since multiplying by $n$ keeps the probability of
failure exponentially small, the property $\{\forall v\in Q,\
|\ell(v)|\leq n^{1/4+\e}\}$ is true a.a.s.\ with exponential
rate. This completes the proof.
\end{proof}

The next theorem generalizes Proposition~\ref{prop:diamQ} to the weighted
case, 
which is needed later on. The analytical part of the proof is  more delicate since the system specifying weighted labelled trees needs two lines, and has to
be transformed to a  one-line equation in order to apply Lemma~\ref{lem:height}.

\begin{theorem}\label{theo:diamQx}
Let $0<a<b$. The diameter of a random rooted
quadrangulation with $n$ faces and weight $x$ at black vertices
 is, a.a.s.\ with exponential rate, in the interval $ (n^{1/4-
\e}, n^{1/4+\e})$, uniformly over $x\in[a,b]$.
\end{theorem}
\begin{proof}
A bicolored labelled tree is called \emph{black-rooted} (resp. white-rooted) if the root-vertex is black (resp. white).
In a bicolored labelled tree the \emph{white-black}
depth of a vertex $v$  is defined as the number of
edges going from a white to a black vertex
on the path from the root-vertex to $v$, and the \emph{white-black height}
is defined as the maximum  of the white-black
depth over
all vertices.  We use here a decomposition of a bicolored
labelled tree into monocolored components (the components are obtained by removing
the bicolored edges), each such component being a plane tree. Let $f(z)$ (resp. $g(z)$)
be the weighted generating function of black-rooted (resp. white-rooted) bicolored labelled trees,
where $z$ marks the number of vertices, and where each tree $\tau$ with $i$ black vertices
has weight $w(\tau)=x^i$. Let $T(z)$
be the series counting rooted plane trees according to edges,
$T(z)=1/(1-zT(z))$. A tree counted by $f(z)$
is made of a monochromatic component (a rooted plane tree) where in each corner
one might insert a sequence of trees counted by $g(z)$; in addition each time one inserts a tree counted by $g(z)$
one has to choose if the label increases or decreases along the corresponding black-white edge.
Since a rooted plane tree with $k$ edges has $2k+1$ corners
and $k+1$ vertices, we obtain
$$
f(z)=\frac{xz}{1-2g(z)}T\left(\frac{xz}{(1-2g(z))^2}\right).
$$
Similarly
$$
g(z)=\frac{z}{1-2f(z)}T\left(\frac{z}{(1-2f(z))^2}\right).
$$
Hence the series $y=f(z)$ satisfies the equation $y=F(z,y)$, where $F(z,y)$ is
expressed by
\begin{equation}\label{eq:fg}
\renewcommand{\arraystretch}{2}
\begin{array}{ll}
F(z,y)=\ds\frac{xz}{1-2G(z,y)}T\left(\frac{xz}{(1-2G(z,y))^2}\right), \\ G(z,y)=\ds\frac{z}{1-2y}T\left(\frac{z}{(1-2y)^2}\right).
\end{array}
\end{equation}
In addition, the white-black height is a height-parameter for this
system.

\vspace{.2cm}

\noindent{\bf Claim.} \emph{The system~\eqref{eq:fg} is admissible.}

\vspace{.2cm}

\noindent\emph{Proof of the claim.} Let $\rho$ be the singularity of $f(z)$ and
 $\tau=f(\rho)$.
Let us prove first that $G(z,y)$  is analytic at $(\rho,\tau)$. Note
that $\tau<1/2$, otherwise there would be $z_0\leq\rho$ such that
$f(z_0)=1/2$, in which case $g(z)$ (and $f(z)$ as well) would diverge to $\infty$ as $z\to z_0^-$,
contradicting the fact that $f(z)$ converges for $0\leq |z|\leq \rho$.
The other possible cause of singularity is
$\rho/(1-2\tau)^2$ being a singularity of $T(z)$.
We use the symbol $\succeq$
for coefficient-domination, i.e., $A(z)\succeq B(z)$ if $[z^n]A(z)\geq [z^n]B(z)$ for all $n\geq 0$.
Clearly we have
$$
f(z)\succeq 2xzg(z),\ \ \ \ g'(z)\succeq 2zf'(z)T'\left(\frac{z}{(1-2f(z))^2}\right),
$$
hence
$$
f'(z)\succeq 4xz^2f'(z)T'\left(\frac{z}{(1-2f(z))^2}\right).
$$
As a consequence,  
$$
T'\left(\frac{z}{(1-2f(z))^2}\right)\leq\frac1{4xz^2}, \quad  \mathrm{as}\ z\to\rho^-.
$$
Since $T'(u)$ diverges at its singularity $1/4$, we have
 $\rho/(1-2\tau)^2\neq 1/4$,  otherwise there would
be the contradiction that the left-hand side
 diverges whereas the  right-hand side, which is larger, converges as $z\to\rho^-$.
Hence $T$ is analytic at $\rho/(1-2\tau)^2$, which ensures that $G(z,y)$ is analytic at $(\rho,\tau)$.
One proves similarly that
$F(z,y)$ is also analytic at $(\rho,\tau)$. \qedclaim

\vspace{.2cm}

The claim, combined with Lemma~\ref{lem:height},
ensures that the white-black height of a random black-rooted bicolored labelled tree with $n$ edges and weight $x$ at
black vertices ($x\in[a,b]$) is in
$(n^{1/2-\e},n^{1/2+\e})$ a.a.s.\ with exponential rate. In addition,
the chain of calculations in~\cite{FlGa93} to prove
Lemma~\ref{lem:height} is easily seen to be uniform in $x\in[a,b]$.
A similar analysis ensures that the \emph{white-black} height of  a random white-rooted bicolored labelled tree with $n$ edges and weight $x$ at
black vertices is in
$(n^{1/2-\e},n^{1/2+\e})$ a.a.s.\ with exponential rate.
Hence, overall, the white-black height of a random bicolored tree (either black-rooted or white-rooted)
with $n$ edges and weight $x$ at black vertices is in
$(n^{1/2-\e},n^{1/2+\e})$ a.a.s.\ with exponential rate.

Now the proof can be concluded in a similar way as in
Proposition~\ref{prop:diamQ}. Define the bicolored depth of a
vertex $v$ from the root as the number of bicolored edges on the path
from the root to $v$, and define the bicolored height as the maximum of
the bicolored depth over all vertices in the tree. Note that the
bicolored depth $d(v)$ and the white-black depth $d'(v)$ of a
vertex $v$ satisfy the inequalities $2d'(v)-1\leq d(v)\leq 2d'(v)+1$,
so the bicolored height is in $(n^{1/2-\e},n^{1/2+\e})$ a.a.s.\ with
exponential rate, uniformly over $x\in[a,b]$.  Similarly as in
Proposition~\ref{prop:diamQ}, this ensures that $\ell_{max}-\ell_{min}$ is in
$(n^{1/4-\e},n^{1/4+\e})$ a.a.s.\ with exponential rate. And the
uniformity over $x\in[a,b]$ follows from the uniformity over
$x\in[a,b]$ for the height.

Finally, using the bijection of Theorem~\ref{theo:bij_quad_trees},
the property that $\ell_{max}-\ell_{min}$ is in
$(n^{1/4-\e},n^{1/4+\e})$ a.a.s.\ with exponential rate is transferred
to the property that the diameter of a random quadrangulation with $n$
faces (with a marked vertex and a marked edge) and weight $x$ at each black vertex is in $(n^{1/4-\e},n^{1/4+\e})$ a.a.s.\ with exponential rate.
There is however a last subtlety to deal with, namely that in the bijection from bicolored
labelled trees to  quadrangulations with a marked vertex and a marked edge,
the number of black vertices in the tree
corresponds either to the number of black vertices or to the number of black vertices plus one in the associated quadrangulation.
So the weighted distribution (weight $x$ at black vertices) on bicolored labelled trees with $n$ edges is not exactly transported
to the weighted distribution (weight $x$ at black vertices) on rooted quadrangulations with $n$ faces
and a secondary marked vertex.
However, since the inaccuracy on the number of black vertices in the quadrangulation is by at most one,
 the transported weighted distribution is biased by at most $x$, so the large deviation result also holds under the (perfectly) weighted
distribution for quadrangulations~\footnote{The color of the marked vertex would be a delicate issue if we were trying to prove an explicit limit
distribution (instead of large deviation results) for the diameter.}.
\end{proof}

\subsection{Maps}

We use here the bijection of Section~\ref{sec:bij_quad_maps}
to get a diameter estimate
for random maps from a diameter estimate for random quadrangulations.
First we need the following lemma.

\begin{lemma}\label{lem:root-face-degree-planar-maps}
Let $M(z,u)$ be the generating function of rooted maps,
where $z$ marks the number of edges, $u$ marks
the degree of the outer face, and with weight $x$ at each vertex. Let $\rho$ be the radius
of convergence of $M(z,1)$ (note that $\rho$ depends on $x$).
Then there is $u_0>1$ such that $M(\rho,u_0)$ converges.
In addition for $0<a<b$, the value of $u_0$ can be chosen uniformly over $x\in[a,b]$,
and $M(\rho,u_0)$ is uniformly bounded over $x\in[a,b]$.
\end{lemma}
\begin{proof}
The result follows easily from a bijection by Bouttier, Di Francesco and Guitter~\cite{BoDFGu02} 
between vertex-pointed planar maps and a certain family of decorated
trees called \emph{mobiles}, such that each face of degree $i$ in the map corresponds
to a (black) vertex of degree $i$ in the mobile. Thanks to this bijection, the generating
function $M^{\circ}(z,u)$ of rooted maps with a secondary marked vertex (where
again $z$ marks the number of edges and $u$ marks the root-face degree)
equals the generating function of rooted mobiles where $z$ marks half
the total degree of (black) vertices and $u$ marks the root-vertex degree.
Since mobiles (as rooted trees) satisfy an explicit decomposition at the root,
the series $M^{\circ}(z,u)$ is easily shown to have, for any $x>0$, a square-root singular development of the
form
$$
M^{\circ}(z,u)=a(z,u)-b(z,u)\sqrt{1-z/\rho},
$$
valid in a neighborhood of $(\rho,1)$, with $a(z,u)$ and $b(z,u)$
 analytic in the parameters $z,u,x$.
Hence the statement holds for $M^{\circ}(z,u)$. Since $M^{\circ}(z,u)$ dominates
$M(z,u)$ coefficient-wise, the statement also holds for $M(z,u)$.
\end{proof}

\begin{theorem}\label{theo:DM}
Let $0<a<b$. The diameter of a random rooted map with $n$ edges and
weight $x$ at the vertices is in the interval $(n^{1/4- \e},n^{1/4+\e})$ a.a.s.\ with exponential rate, uniformly over $x\in[a,b]$.
\end{theorem}
\begin{proof}
The first important observation is that the bijection of Section~\ref{sec:bij_quad_maps} transports the
weighted (weight $x$ at black vertices) distribution on rooted
quadrangulations with $n$ faces to the weighted (weight $x$ at vertices)
distribution on rooted maps with $n$ edges. Let $M$ be a random
rooted map with $n$ edges and let $Q$ be the associated rooted
quadrangulation (with $n+2$ vertices). Since $\D(Q) \le 2\ \! \D(M)$, the
diameter of $M$ is at least $n^{1/4-\e}$ a.a.s.\ with exponential
rate. The upper bound is proved from the inequality
$\D(M)\leq \D(Q)\cdot\Delta(M)$, where $\Delta(M)$ is the maximal
face degree in $M$.
Together with Lemma~\ref{lem:tail}, Lemma~\ref{lem:root-face-degree-planar-maps}
ensures that the root-face degree $\delta(M)$
 in a random rooted planar map with $n$ edges and weight $x$ at vertices
has exponentially fast decaying tail.
The probability distribution of $\delta(M)$ is the same if $M$ is bi-rooted
(i.e., has two roots that are possibly equal, the root-face being
the face incident to the primary root).
When exchanging the secondary root with the primary root, the root-face can
be seen as a face $f$ taken at random under the distribution $P(f)=\mathrm{deg}(f)/(2n)$.
Thus $\delta(M)$ is distributed as the degree of the (random) face $f$.
Hence $$P(\delta(M)\geq k)\geq \frac{k}{2n}P(\Delta(M)\geq k),$$
so that $\Delta(M)\leq n^{\e}$
a.a.s.\ with exponential rate. We conclude from~\eqref{eq:diamMQ}
that the diameter of $M$ is at most $n^{1/4+\e}$ a.a.s.\ with exponential rate.
The uniformity in $x\in [a,b]$ follows from the uniformity in $x\in[a,b]$ in
Theorem~\ref{theo:diamQx} and
Lemma~\ref{lem:root-face-degree-planar-maps}.
\end{proof}

\subsection{2-connected maps}\label{sec:2connmaps}
Let $x>0$.
 Denote by
$M(z)$ (resp. $C(z)$) the weighted generating function of rooted connected (resp. 2-connected)
maps according to edges and with weight $x$ at non-root vertices.
Since a core with $n$ edges has $2n$ corners where to insert (possibly empty) rooted maps, this decomposition yields
\begin{equation}\label{eq:MCeq}
M(z)=\sum_{n\geq 0}z^n\sum_{\tau\in\cC_n} \big(1+M(z)\big)^{2n}=C(H(z)),\ \ \mathrm{where}\ H(z)=z(1+M(z))^2.
\end{equation}


An important property of the core-decomposition is that it
preserves the
distribution with weight $x$ at vertices.
Precisely, let $M$ be a random rooted map with $n$ edges and weight $x$ at vertices.
Let $C$ be the core of $M$ and let $k$ be its size. Let
$M_1,\ldots,M_{2k}$ be the pieces
of $M$, and $n_1,\ldots,n_{2k}$ their sizes. Then, conditioned to having size $k$,
$C$ is a random rooted 2-connected map with $k$ edges and weight $x$ at vertices;
and conditioned to having size
$n_i$, the $i$th piece $M_i$ is a random rooted map with $n_i$ edges and
weight $x$ at vertices.

\begin{lemma}\label{lem:core_pieces}
Let $0<a<b$, and let $x\in[a,b]$.
Let $\rho$ be the radius of convergence of $z\mapsto M(z)$ ($M(z)$ gives weight $x$ to vertices).
Following~\cite{BaFlScSo01}, define
$$
\alpha=\frac{H(\rho)}{\rho H'(\rho)}.
$$
Let $n\geq 0$, and
let $M$ be a random rooted map with $n$ edges and weight $x$ at vertices.
Let $X_n=|C|$ be the size of the core
 of $M$, and let $M_1,\ldots,M_{2|C|}$ be the pieces of $M$.
Then
$$
P\left(X_n=\lfloor \alpha n\rfloor,\ \mathrm{max}(|M_i|)\leq n^{3/4}\right)\sim P\left(X_n=\lfloor \alpha n\rfloor\right)=\Theta\left(n^{-2/3}\right)
$$
uniformly over $x\in[a,b]$.
\end{lemma}
\begin{proof}
The statement
$P(X_n=\lfloor \alpha n\rfloor)=\Theta(n^{-2/3})$ uniformly over $x\in[a,b]$ 
follows from~\cite{BaFlScSo01}. So what
we have to prove is that
$P(X_n=\lfloor \alpha n\rfloor,\ \mathrm{max}(|M_i|)> n^{3/4})=o(n^{-2/3})$
uniformly over $x\in[a,b]$.






\noindent{\bf Claim.}  \emph{Given a fixed $\delta>0$,
we have for $i> n^{2/3+\delta}$
$$
P(X_n=\lfloor \alpha n\rfloor,\ \ |M_1|=i)=O(\exp(-n^{\delta/2})).
$$
}

\noindent\emph{Proof of the claim.} Let $a_m$ be the number of rooted maps
and $c_m$ the number of rooted 2-connected maps with $m$ edges.
It follows from the (algebraic) generating function expressions~\cite{Tu63,Ar85} that these
numbers have the asymptotic estimates $a_m\sim c\rho^{-m}m^{-5/2}$,
$c_m\sim c'\sigma^{-m}m^{-5/2}$.
Equation~\eqref{eq:MCeq} implies
$$
P(X_n=k)=c_k\frac{[z^n]H(z)^k}{a_n}.
$$
It is proved in~\cite[Theorem 1 (iii)-(b)]{GaWo99},  (and the bounds are easily checked to hold uniformly over $x\in[a,b]$)
 that for $k\geq \alpha n+n^{2/3+\delta}$,
\begin{equation}\label{eq:Hzk}
[z^n]H(z)^k=
O(\sigma^k\rho^{-n}\exp(-n^{\delta})).
\end{equation}
Let $k_0=\lfloor \alpha n\rfloor$ and let $n^{2/3+\delta}<i\leq n-k_0$. We have
\begin{eqnarray*}
P(X_n=k_0,\ |M_1|=i) &= & c_{k_0}\frac{a_i[z^{n-i}]z^{k_0}(1+M(z))^{2k_0-1}}{a_n}\\
&\leq & c_{k_0}\frac{a_i[z^{n-i}]H(z)^{k_0}}{a_n}=O(n^{5/2}\sigma^{-k_0}\rho^{n-i}[z^{n-i}]H(z)^{k_0}).
\end{eqnarray*}
Since $\alpha n/(n-i)\geq \alpha(1+i/n)$, we have $\alpha n\geq \alpha(n-i)+\alpha i(n-i)/n=\alpha(n-i)+\Omega(n^{2/3+\delta})=\alpha(n-i)+\Omega((n-i)^{2/3+\delta})$. 
Hence $k_0=\alpha(n-i)+\Omega((n-i)^{2/3+\delta})$, 
so~\eqref{eq:Hzk} ensures that for any fixed $\delta'<\delta$, 
$$
[z^{n-i}]H(z)^{k_0}=O(\sigma^{k_0}\rho^{-n+i}\exp(-(n-i)^{\delta'})).
$$
Hence, for $i>n^{2/3+\delta}$, and for any fixed $\delta''<\delta'$, 
$$
P(X_n=k_0,\ |M_1|=i)=O(\exp(-(n-i)^{\delta''})),
$$
so that $P(X_n=k_0,\ |M_1|=i)=O(\exp(-n^{\delta/2}))$.
\qedclaim

\noindent The claim implies that
 $P(X_n=\lfloor \alpha n\rfloor,\ |M_1|> n^{2/3+\delta})=O(n\exp(-n^{\delta/2}))$,
 and by symmetry the same estimate holds for each piece $M_i$.
 As a consequence $P(X_n=\lfloor \alpha n\rfloor,\ \mathrm{max}(|M_i|)> n^{2/3+\delta})=
O(n^2\exp(-n^{\delta/2}))=O(\exp(-n^{\delta/3}))$. Hence
$$
P(X_n=\lfloor \alpha n\rfloor, \ \mathrm{max}(|M_i|)\leq n^{2/3+\delta})\sim P(X_n=\lfloor \alpha n\rfloor)=\Theta(n^{-2/3}).
$$
This concludes the proof, taking $\delta=3/4-2/3=1/12$.
\end{proof}

In~\cite{BaFlScSo01} the authors  show that $n^{2/3}P(X_n =\lfloor \alpha n\rfloor)$ converges; they even prove that $(X_n-\alpha n)/n^{2/3}$ converges in law.
Lemma~\ref{lem:core_pieces} just makes sure that the asymptotic estimate of
 $P(X_n =\lfloor \alpha n\rfloor)$ is the same under the additional
condition that all pieces are of size at most $n^{3/4}$ (more generally,
 under the condition that all pieces are of size at most $n^{2/3+\delta}$,
  for any $\delta>0$).
A closely related result proved in~\cite{GaWo99}
is that, for any fixed $\delta>0$,
 there is a.a.s.\ no piece of size larger than $n^{2/3+\delta}$ provided
 the core has size larger than $n^{2/3+\delta}$.

\begin{theorem}\label{theo:diam2c}
For $0<a<b$, the diameter of a random rooted 2-connected map with
$n$ edges and weight $x$ at vertices  
 is, a.a.s.\ with
exponential rate, in the interval $    (n^{1/4- \e}, n^{1/4+\e}) $,
uniformly over $x\in[a,b]$.
\end{theorem}

\begin{proof}
Let $M$ be a rooted map with $n$ edges and weight $x$ at vertices. Denote by $C$
the core of $M$ and by $(M_i)_{i\in[1..2|C|]}$ the pieces of $M$.
Since the event $\{|C|=\lfloor \alpha n\rfloor\}$ has polynomially small probability (order $\Theta(n^{-2/3})$,
as shown in~\cite{BaFlScSo01}), and since the event
 $\D(M)\leq n^{1/4+\e}$ holds a.a.s.\ with exponential rate,
the event $\D(M)\leq n^{1/4+\e}$, knowing that $|C|=\lfloor \alpha n\rfloor$, also
holds a.a.s.\ with exponential rate. Since $\D(M)\geq \D(C)$, we conclude that
for $C$ a random 2-connected map with $\lfloor \alpha n\rfloor$ edges and weight $x$ at vertices,
$\D(C)\leq n^{1/4+\e}$ a.a.s.\ with exponential rate. Of course the same holds for $C$
a random rooted 2-connected map with $n$ edges and weight $x$ at vertices.
This yields the a.a.s.\ upper bound on $\D(C)$.

To prove the lower bound, we use Lemma~\ref{lem:core_pieces}, which
ensures that the event
$$\{|C|=\lfloor \alpha n\rfloor,\
\mathrm{max}(|M_i|)\leq n^{3/4}\}$$
occurs with polynomially small probability,
precisely $\Theta(n^{-2/3})$.
We claim that, under the condition that $\mathrm{max}(|M_i|)\leq n^{3/4}$, then
 $\mathrm{max}(\D(M_i))\leq n^{1/5}$ a.a.s.\ (in $n$) with exponential rate.
Indeed, consider a piece $M_i$ of size $n_i$.
When $n_i\leq n^{1/5}$, $\D(M_i)\leq n^{1/5}$ trivially.
Moreover, Theorem~\ref{theo:DM} implies that, for $\delta>0$ small enough,
$P(\D(M_i)>n_i^{1/4+\delta})\leq\exp(-n_i^{c\delta})$
for some $c>0$. Hence when $n^{1/5}\leq n_i\leq n^{3/4}$,
$P(\D(M_i)>n^{3/4(1/4+\delta)})\leq\exp(-n^{c\delta/5})$, and we can take $\delta$ small enough
so that $3/4(1/4+\delta)\leq 1/5$.
Hence, when $n_i\leq n^{3/4}$, the event $\D(M_i)> n^{1/5}$ has exponentially small probability in $n$ (meaning, in $O(\exp(-n^{\alpha})$ for some
$\alpha>0$),
and the same holds for $\mathrm{max}(\D(M_i))$.
 Hence
$$
\mathbb{P}(\{|C|=\lfloor \alpha n\rfloor,\ \mathrm{max}(\D(M_i))\leq n^{1/5}\})\sim \mathbb{P}(\{|C|=\lfloor \alpha n\rfloor\})=\Theta(n^{-2/3}).
$$
In other words the event $\{|C|=\lfloor \alpha n\rfloor,\ \mathrm{max}(\D(M_i)\leq n^{1/5}\}$ occurs with polynomially small probability.
In that case, since $\D(C)\geq \D(M)-2\ \!\mathrm{max}(\D(M_i))$, and since the event $\D(M)<n^{1/4-\e}$ occurs a.a.s.\ with exponential rate,
we conclude that $\D(C)\geq n^{1/4-\e}-2n^{1/5}$ holds a.a.s.\ with exponential rate under the event
$\cE=\{|C|=\lfloor \alpha n\rfloor,\ \mathrm{max}(\D(M_i)\leq n^{1/5}\}$.
Since~$\cE$ occurs with probability $\Theta(n^{-2/3})$  and since
$n^{1/5}=o(n^{1/4-\e})$ for $\e$ small enough, we conclude (similarly
as in the proof of Theorem~\ref{theo:diam2c}) that for $C$
a random 2-connected map with $\lfloor \alpha n\rfloor$ edges and weight $x$ at vertices,
we have $\D(C)\geq n^{1/4-\e}$ a.a.s.\ with exponential rate. The same holds for $C$
a random rooted 2-connected  with $n$ edges and weight $x$ at vertices.

The uniformity
in $x\in[a,b]$ of the bounds follows from the uniformity in $x$ in
Theorem~\ref{theo:DM} and
Lemma~\ref{lem:core_pieces}.
\end{proof}

\subsection{3-connected maps}
In the following we assume 3-connected maps (and 3-connected planar graphs)
to have at least $4$ vertices, so the smallest 3-connected planar graph
is $K_4$.
We use here the plane network decomposition (Section~\ref{sec:3ccore})
to carry the diameter concentration property from 2-connected to
3-connected maps. For $x>0$,
call $N(z)$ (resp. $\wh{N}(z)$) the weighted generating functions
---weight $x$ at vertices not
incident to the root-edge--- of
plane networks (resp. plane networks with a 3-connected core),
where $z$ marks the number of edges. Note that $N(z)$ is very close
to the generating function $C(z)$ of rooted 2-connected maps
with weight $x$ at non-root vertices
and with $z$ marking the number of edges:
$$
C(z)=z+xz+xzN(z),
$$
where the first two terms in the right-hand side stand for the two 2-connected maps with a single
edge, either a loop or a link between two distinct vertices.
Call $T(z)$ the weighted generating function of rooted 3-connected maps, with weight
$x$ at vertices not incident to the root-edge, and with $z$ marking the number of non-root edges.
Clearly, the weighted generating function
 $S(z)$ of plane networks decomposable as a sequence
of plane networks satisfies $S(z)=(N(z)-S(z))xN(z)$, hence $S(z)=xN(z)^2/(1+xN(z))$.
Similarly  the weighted generating function $P(z)$ of parallel plane networks satisfies
$P(z)=(N(z)-P(z))N(z)$, so that $P(z)=N(z)^2/(1+N(z))$. Hence

\begin{equation}\label{eq:planeNetworks}
N(z)=S(z)+P(z)+\wh{N}(z),
\end{equation}
where
$$
S(z)=\frac{xN(z)^2}{1+xN(z)},\quad  P(z)=\frac{N(z)^2}{1+N(z)},\quad  \wh{N}(z)=T(N(z)).
$$
An important remark is that a random plane network $C$
with $n$ edges and weight $x$ at vertices  can be seen as a random $2$-connected
map with $n+1$ edges, weight $x$ at vertices, and where the root-edge has been deleted.
Similarly as in Section~\ref{sec:2connmaps}, for a random
plane network $N$ with $n$ edges and weight $x$ at vertices,
and conditioned to have  a 3-connected core $T$ of size $k$, $T$ is a random rooted 3-connected map
with $k$ edges and weight $x$ at vertices;
and each piece $C_e$ conditioned to have a given
 size $n_e$ is a random plane network with $n_e$ edges and weight
 $x$ at vertices.

For proving the diameter estimate for 3-connected maps, we need
the following lemma, ensuring that the root-face degree of a random 2-connected
map is small.

\begin{lemma}\label{lem:root-face-2conn}
Let $C(z,u)=\sum_{n,k}c_{n,k}z^nu^k$ 
be the generating function of  rooted 2-connected maps,
  where $z$ marks the number of edges, $u$ marks
the root-face degree, and with weight $x$ at each non-root vertex. Let~$R$ be the radius
of convergence of $C(z,1)$. Then there is $v_0>1$ such that $C(\rho,v_0)$ converges.
In addition for $0<a<b$, the value of $v_0$ can be chosen uniformly over $x\in[a,b]$,
and $C(z,v_0)$ is uniformly bounded over $x\in[a,b]$.
\end{lemma}
\begin{proof}
The result has been established for arbitrary rooted maps in Lemma~\ref{lem:root-face-degree-planar-maps}.
To prove the result for 2-connected maps, we rewrite
Equation~\eqref{eq:MCeq}
taking account of the root-face degree.
Recall that a rooted map $\gamma$ is obtained from a rooted 2-connected 
map $\kappa$ 
where a rooted map (allowing for the one-vertex map) is inserted in each corner;
call $k$ the root-face degree of $\kappa$ and $\gamma_1,\ldots,\gamma_k$ the maps
inserted in the root-face corners of $\kappa$. 
If $d(M)$ denotes the root-face degree of a rooted map $M$, then clearly
$$
d(\gamma)=k+d(\gamma_1)+\cdots+d(\gamma_{k}).
$$
Hence, (with $M(z):=M(z,1)$):
$$
M(z,u)=\sum_{n,k}c_{n,k}u^k(1+M(z))^{2n-k}(1+M(z,u))^k,
$$
so that 
$$
M(z,u)=C\left(z(1+M(z))^2,u{1+M(z,u)\over 1+M(z)}\right).
$$
Since the composition scheme is ``critical'' \cite{BaFlScSo01}, it is known that, if $\rho$ denotes the radius of convergence
of $M(z,1)$, then $R=\rho\cdot(1+M(\rho))^2$ is the radius of convergence of $C(z,1)$.
Hence, since $M(\rho,u_0)$ converges, $C(R,v_0)$ converges for
$v_0=u_0(1+M(\rho,u_0))/(1+M(\rho))>1$.  
The uniformity statement for $C(z,u)$ (for $x\in[a,b]$)
 follows from
the uniformity statement for $M(z,u)$, established in Lemma~\ref{lem:root-face-degree-planar-maps}, and the fact that $v_0$
is uniformly bounded away from $1$ when $x$ lies in a compact interval.
\end{proof}

\begin{theorem}\label{theo:3conn_maps}
Let $0<a<b$. The diameter of a random 3-connected map with $n$ edges
with weight $x$ at vertices is, a.a.s.\ with exponential rate, in
the interval $    (n^{1/4- \e}, n^{1/4+\e})$, uniformly over
$x\in[a,b]$.
\end{theorem}
\begin{proof}
Let $\rho$ be the radius of convergence (depending on the weight $x$ at vertices)
of $N(z)$, which is the same as the radius of
convergence of $C(z)=z+xz+xzN(z)$. And let
$$
\alpha=\frac{N(\rho)}{\rho N'(\rho)}.
$$
Again the results in~\cite{BaFlScSo01} ensure that, for a random plane network $C$
with $n$ edges and weight $x$ at vertices,
the probability of having a 3-connected core $T$ of size
$\lfloor \alpha n\rfloor$
is $\Theta(n^{-2/3})$, hence polynomially small, whereas the probability that $\D(C)>n^{1/4+\e}$
 is exponentially small. Since $\D(C)\geq \D(T)$, and since $T$ is a random rooted $3$-connected map
with $k=\lfloor \alpha n\rfloor$ edges and weight $x$ at vertices, we conclude that
$\D(T)\leq n^{1/4+\e}$
a.a.s.\ with exponential rate.
 For the lower bound we look at the second inequality in~\eqref{eq:lowerDC}:
$$
 \D(C)\leq \D(T)\cdot\mathrm{max}_{e\in T}(d_e)+2\ \!\mathrm{max}_{e\in T}(\D(C_e)),
$$
where for each edge $e$ of $T$, $C_e$ denotes the piece substituted at $e$ and $d_e$
denotes the root-face degree of $C_e$.

Lemma~\ref{lem:tail} and Lemma~\ref{lem:root-face-2conn}
ensure that  the distribution of the root-face degree of a random
rooted 2-connected map has exponentially fast decaying tail.
Hence $\mathrm{max}_{e\in T}(d_e)\leq n^{\e}$ a.a.s.\ with exponential rate.
 Moreover, in the same way as in Lemma~\ref{lem:core_pieces},
one can show that
the probability of the event
$\cE=\{|T|=\lfloor \alpha n\rfloor,\ \mathrm{max}(|C_e|)\leq n^{3/4}\}$
is $\Theta(n^{-2/3})$. Since  $\mathrm{max}_{e\in T}(d_e)\leq n^{\e}$
and $\D(C)\geq n^{1/4-\e}$ a.a.s.\ with exponential rate, Equation~\eqref{eq:lowerDC} easily implies that,
conditioned on $\cE$, $\D(T)\geq n^{1/4-\e}$ a.a.s.\ with exponential rate.
Since $\cE$ occurs with polynomially small probability, we conclude that
$\D(T)\geq n^{1/4-\e}$ a.a.s.\ with exponential rate.
Finally the uniformity of the estimate over $x\in[a,b]$ follows from the
uniformity over $x\in[a,b]$ in Theorem~\ref{theo:diam2c} and in Lemma~\ref{lem:root-face-2conn}.
\end{proof}

\section{Diameter estimates for families of graphs}\label{sec:diam_graphs}
We now establish estimates (all of the form $\D(G)\in(n^{1/4-\e},n^{1/4+\e})$ a.a.s. 
with exponential rate)  for the diameter of random graphs in families
of \emph{unembedded} planar graphs. We establish first an estimate 
for 3-connected planar graphs (equivalent to 3-connected maps by Whitney's theorem),
then derive from it an estimate for 2-connected planar graphs (which have a decomposition, at edges, 
into 3-connected components), and finally derive from it an estimate for connected
planar graphs (which have a decomposition, at vertices, 
into 2-connected components). 
Since the graphs are unembedded,
it is necessary to \emph{label} them to avoid symmetry issues (in contract, for maps,  
rooting, i.e., marking and orienting an edge, is enough). One can choose
to label either the vertices or the edges. For our purpose it is more convenient
to label 3-connected and 2-connected planar graphs at \emph{edges}
(because the decomposition into 3-connected components occurs at edges);
then relabel 2-connected planar graphs at vertices and label also connected
planar graphs at vertices (because the decomposition into 2-connected components
occurs at vertices). 

\subsection{3-connected planar graphs}

For the time being
we need 3-connected graphs labelled at the edges (this is enough to avoid symmetries).
The number of edges is denoted $m$, and $n$ is reserved for the number of vertices.
By Whitney's theorem, 3-connected planar graphs with at least $4$ vertices have two 
  embeddings on the oriented sphere (which are mirror of each other).
Hence Theorem~\ref{theo:3conn_maps} gives:

\begin{theorem}\label{theo:D3cgraph}
Let $0<a<b$. The diameter of a random 3-connected planar graph with
$m$ edges and weight $x$ at vertices is, a.a.s.\ with
exponential rate, in the interval $(m^{1/4- \e}, m^{1/4+\e})$, uniformly over $x\in[a,b]$.
\end{theorem}

\subsection{Planar networks}

Before handling 2-connected planar graphs we treat the closely related family
of \emph{planar networks}. A \emph{planar network} is a connected simple planar graph
with two marked vertices called the poles, such that adding an edge
between the poles, called the root-edge,
makes the graph 2-connected. First it is convenient
to consider planar networks as labelled at the edges.

\begin{theorem}\label{theo:Nm}
Let $0<a<b$.
The diameter of a random planar network with $m$ labelled edges and weight $x$ at vertices is, a.a.s.\ with exponential rate, in the interval
$$
    (m^{1/4- \e}, m^{1/4+\e}),
$$
uniformly over $x\in[a,b]$.
\end{theorem}
The proof, which is quite technical, is delayed to Section~\ref{sec:proof_theo_delayed};
it relies on the decomposition into \emph{$3$-connected components} described in Section~\ref{subsec:2to3}
and the inequalities~\eqref{eq:diam3c}.
The proof of Theorem~\ref{theo:cx} in the next section, which relies
on the decomposition into \emph{2-connected components} gives a good idea
(with less technical details), of the different steps
needed to prove Theorem~\ref{theo:Nm}.

\begin{lemma}\label{lem:Nnm}
Let $1<c<d<3$. Let $N_{n,m}$ be a planar network with $n$ vertices and $m$
labelled edges, taken uniformly at random. Then
$\D(N_{n,m})\in(n^{1/4-\e},n^{1/4+\e})$ a.a.s.\ with exponential rate,
uniformly over $m/n\in[c,d]$.
\end{lemma}
\begin{proof}
Let $\mu=m/n$. For $x>0$, let $X_m$ be the number of vertices of a
random  planar network $N$ with $m$ edges and weight $x$ at vertices.
The results in~\cite{BeGa} ensure that
there exists $x_{\mu}>0$ such that, for $x=x_{\mu}$,
$P(X_m=n)=\Theta(m^{-1/2})$, uniformly over $\mu\in[c,d]$.
In addition $x_{\mu}$ is a continuous function of $\mu$, so
it maps $[c,d]$ into a compact interval.
Therefore, Theorem~\ref{theo:Nm} implies that, for $x=x_\mu$,
 $\D(N)\in [m^{1/4-\e},m^{1/4+\e}]$ a.a.s.\ with exponential rate
 uniformly over $\mu\in[c,d]$. Since $P(X_m=n)=\Theta(m^{-1/2})$, uniformly over $\mu\in[c,d]$, we conclude that the event
 $\D(N)\in [m^{1/4-\e},m^{1/4+\e}]$, conditioned on $X_m=n$, holds
 a.a.s.\ with exponential rate uniformly over $\mu\in[c,d]$, which concludes the proof
 (note that the distribution of $N$ conditioned on $X_m=n$ is the
 uniform distribution on planar networks with $m$ edges and $n$ vertices).
\end{proof}

The proof of Lemma~\ref{lem:Nnm} is the only place where uniformity of the
estimates according to $x$ (for $x$ in an arbitrary compact interval)
is needed. In the following, the weight $x$ will be at edges,
and we will not need anymore to check that the statements hold uniformly in $x$ (even though they clearly do).
Another important remark is that planar networks with $n$ vertices and $m$ edges can be labelled
either at vertices or at edges, and the uniform distribution in one case
corresponds to the uniform distribution in the second case. Hence the result
of Lemma~\ref{lem:Nnm} holds for random planar networks with $n$ labelled
 vertices and $m$ unlabelled edges.

\begin{lemma}\label{lem:Nnx}
Let $x>0$. Let $N$ be a random planar network with $n$ labelled vertices and weight $x$ at edges (which are unlabelled).
Then $\D(N)\in(n^{1/4-\e},n^{1/4+\e})$ a.a.s.\ with exponential rate.
\end{lemma}
\begin{proof}
As shown in~\cite{BeGa}, the ratio $r=\#(edges)/\#(vertices)$ of $N$ is concentrated around some value $\mu=\mu(x)\in(1,3)$.
Precisely, for each $\delta>0$, there is $c= c(\delta)>0$ such that $$\mathbb{P}\{r\notin(\mu-\delta,\mu+\delta)\}\leq \exp(-cn).$$
Take $\delta$ small enough so that $r-\delta>1$ and $r+\delta<3$. Then Lemma~\ref{lem:Nnm} ensures that
$\D(N)\in [n^{1/4-\e},n^{1/4+\e}]$ a.a.s.\ with exponential rate.
\end{proof}

\subsection{2-connected planar graphs}

Planar networks are very closely related to edge-rooted $2$-connected planar graphs.
In fact, an edge-rooted (i.e., with a marked oriented edge) $2$-connected planar graph $G$ 
yields two planar networks: one where the marked edge is kept (otherly stated, doubled and then 
one copy of the marked edge is deleted) and one where the marked edge
is deleted (in the second case the diameter of the planar network might be larger
than the diameter of $G$, however by a factor of at most $2$).
 Consequently the statement of Lemma~\ref{lem:Nnx} also holds for $N$ a random
edge-rooted $2$-connected planar graph with $n$ (labelled) vertices and weight $x$ at edges. 
And the statement still holds for a random $2$-connected planar graph (unrooted) with $n$ vertices,
since the number of edges can vary only from $n$ to $3n$
(hence the effect of unmarking a root-edge biases the distribution
by a factor of at most $3$). We obtain:

\begin{theorem}\label{theo:2c}
Let $x>0$.
The diameter of a random 2-connected planar graph with $n$ vertices and weight $x$ at edges
is, a.a.s.\ with exponential rate, in the interval $(n^{1/4- \e},
n^{1/4+\e})$.
\end{theorem}

\subsection{Connected planar graphs}\label{sec:1-c}

Here we deduce from Theorem~\ref{theo:2c} that a random connected
planar graph with $n$ vertices has diameter in
$(n^{1/4-\e},n^{1/4+\e})$ a.a.s.\ with exponential rate. We use the block
decomposition presented in Section~\ref{subsec:1to2}, and the inequality~\eqref{eq:ineq_block}.
Again an important point is that if $C$ is a random connected planar graph with $n$ vertices
and weight $x$ at edges, then each block $B$ of size $k$ in
$C$ is a random $2$-connected planar graph with $k$ vertices and weight $x$ at edges.
Note that, formulated on pointed graphs (i.e., graphs with a marked vertex), 
the block-decomposition ensures that a
pointed connected planar graph is obtained as follows: take a collection of
2-connected pointed planar graphs, and merge their marked vertices
into a single vertex; then attach at each non-marked vertex $v$ in
these blocks a pointed connected planar graph $C_v$.
Fix $x>0$.
Call $C(z)$ and $B(z)$ the weighted generating functions, respectively, of connected and 2-connected planar graphs with weight $x$ at edges. Then the decomposition above yields
\begin{equation}\label{eq:FB}
C'(z)=\exp(B'(zC'(z))).
\end{equation}

\begin{lemma}\label{lem:bigblockC}
For $x>0$, a random connected planar graph with $n$ vertices and weight
$x$ at edges has a block of size at least $n^{1-\e}$
a.a.s.\ with exponential rate.
\end{lemma}
\begin{proof}
Denote by $E(z)=zC'(z)$ the series counting pointed connected planar graphs
with weight $x$ at edges.
Note that the functional
inverse of $E(z)$ is $\phi(u)=u\exp(-g(u))$, where $g(u)=B'(u)$.
Call $\rho$ the radius of convergence of $C(z)$ and $R$ the radius of convergence of $B(u)$.
Define $b_i:=[u^i]g(u)$, $g_k(u):=\sum_{i\leq k}b_iu^i$, and call $E_k(z)$ the series
of pointed connected planar graphs where all blocks have size at most $k$. Note that
the probability of a random connected planar graph with $n$ vertices to have all its
blocks of size at most $k$ is $[z^n]E_k(z)/[z^n]E(z)$. Clearly
$$
E_k(z)=z\exp(g_k(E_k(z)),
$$
hence the functional inverse of $E_k(z)$ is $\phi_k(u)=u\exp(-g_k(u))$. Call $\rho_k$ the singularity
of $E_k(z)$. Since $\phi_k(u)$ is analytic everywhere, the singularity at $\rho_k$ is caused by a branch point,
i.e., $\rho_k=\phi_k(R_k)$, where $R_k$ is the unique $u>0$ such that $\phi_k'(u)=0$:
$\phi_k'(u)>0$ for $0<u<R_k$ and $\phi'(u)<0$ for $u>R_k$. According to~\eqref{eq:saddle}, $[z^n]E_k(z)\leq E_k(s)s^{-n}$  for $s<\rho_k$, or equivalently, writing $u=E_k(s)$,
\begin{equation}\label{eq:Fkb}
[z^n]E_k(z)\leq u\phi_k(u)^{-n}\ \ \mathrm{for\ all}\ u\ \mathrm{such\ that}\ \phi_k'(u)>0.
\end{equation}
Define $u_k=R\cdot(1+1/(k\log k))$. Note that
$$
g_k(R)\leq g_k(u_k)\leq \left(\frac{u_k}{R}\right)^kg_k(R).
$$
Since $(u_k/R)^k\to 1$ we have $g_k(u_k)\to g(R)$. Similarly $g_k'(u_k)\to g'(R)$, hence $\phi_k'(u_k)\to \phi'(R)$.
It is shown in~\cite{gimeneznoy} that $a=\phi'(R)$ is strictly positive (i.e., the singularity of $E(z)$
is not due to a branch point), so for $k$ large enough, $\phi_k'(u_k)\geq a/2>0$, i.e.,
the bound~\eqref{eq:Fkb} can be used, giving
$$
[z^n]E_k(z)\leq 2\ \!R\ \!\phi_k(u_k)^{-n}\ \ \mathrm{for\ }k\ \mathrm{large\ enough\ and\ any}\ n\geq 0.
$$
Moreover
$$
\phi_k(u_k)-\rho=\left(\phi_k(u_k)-\phi_k(R)\right)+\left(\phi_k(R)-\phi(R)\right)= a\cdot(u_k-R)+O(k^{-3/2})\sim\frac{a\ \!R}{k\log k},
$$
where $\phi_k(R)-\phi(R)=O(k^{-3/2})$ is due to $g(R)-g_k(R)=O(k^{-3/2})$,
which itself follows from the estimate $b_i=\Theta(R^{-i}i^{-5/2})$ shown
in~\cite{gimeneznoy}. Hence for $k$ large enough and any $n\geq 0$:
$$
[z^n]E_k(z)\leq 2\left(\rho+\frac{a\ \!R}{2k\log k}\right)^{-n}.
$$
Hence, for $k=n^{1-\e}$, $[z^n]E_k(z)=O(\rho^{-n}\exp(-n^{\e/2}))$.
Finally, according to~\cite{gimeneznoy}, $[z^n]E(z)=\Theta(\rho^{-n}n^{-5/2})$, so $[z^n]E_k(z)/[z^n]E(z)=O(\exp(-n^{\e/3}))$.
\end{proof}

\subsection*{Remark.} It is shown in \cite{GNR} and \cite{PS} that a random connected planar graph has a.a.s. a block of linear size, but not with exponential rate. This is the reason for the previous lemma.

\medskip

Lemma~\ref{lem:bigblockC} directly implies that a random connected planar graph with $n$ vertices has diameter
at least $n^{1/4-\e}$. Indeed it has a block of size $k\geq n^{1-\e}$ a.a.s.\ with exponential rate
and since the block is uniformly distributed in size $k$, it has diameter at least $k^{1/4-\e}$
a.a.s.\ with exponential rate.

Let us now prove the upper bound. For this purpose we use the inequality given in Section~\ref{subsec:1to2}:
$$\D(C)\leq (\D(\tau)+1)\cdot\mathrm{max}_i(\D(B_i)),$$
where $C$ denotes a connected planar graph, $\tau$ is the Bv-tree, and the $B_i$'s
are the blocks of $C$.
We show that $\D(\tau)\leq n^{\e}$ a.a.s.\ and that $\mathrm{max}_i(\D(B_i))\leq n^{1/4+\e}$ a.a.s., both
with exponential rate.

To show that $\D(\tau)\leq n^{\e}$ we need a counterpart of Lemma~\ref{lem:height} for   \emph{critical} equations of the form~\eqref{eq:tree} 
(Indeed, note that $y\equiv y(z)=C'(z)$ is solution of $y=F(z,y)$, where $F(z,y)=\exp(B'(zy))$; in addition the height of the Bv-tree,
rooted at the pointed vertex, is a height-parameter of that system.)

\begin{lemma}\label{lem:height2}
Let $\cT$ be a combinatorial class endowed with a weight-function $w(\cdot)$ so that the corresponding (weighted) generating function
$y(z)$ satisfies an equation of the form $y=F(z,y)$ that is critical.

Let $\xi$ be a height-parameter for~\eqref{eq:tree} and let $T_n$ be
taken at random in $\cT_n$ under the weighted distribution in size
$n$. Assume that $[z^n]y(z)=\Omega(n^{-\alpha}\rho^{-n})$ for some $\alpha$.
Then $\xi(T_n)\leq n^{\e}$ a.a.s.\ with
exponential rate.
\end{lemma}
\begin{proof}
For $h\geq 0$ we define the generating function $y_h(z)=\sum_{\tau\in\cT,\ \xi(\tau)\leq h}z^{|\tau|}w(\tau)$, so that
$$
y_h(z)=F(z,y_{h-1}(z)), 
$$
and define $\uy_h(z)=\sum_{\tau\in\cT,\ \xi(\tau)=h}z^{|\tau|}w(\tau)$
(i.e., $\uy_h(z)=y_h(z)-y_{h-1}(z)$). Let $\tau_h=y_h(\rho)$ and $\utau_h=\uy_h(\rho)$. Note that $y(z,u)=\sum_h \uy_h(z)u^h$
is the bivariate generating function of $\cT$ where $z$ marks the size and $u$ marks the height.
For $h>0$ we have
$$
\tau_{h+1}-\tau_{h}=F(\rho,\tau_{h})-F(\rho,\tau_{h-1})=F_y(\rho,u_h)\cdot(\tau_h-\tau_{h-1}),\ \ \mathrm{for\ some\ }u_h\in[\tau_{h-1},\tau_h].
$$
Since $\tau_h$ converges to $\tau$ as $h\to\infty$, $u_h$ also converges to $\tau$, hence $F_y(\rho,u_h)$ converges to $F_y(\rho,\tau)<1$.
Consequently $\utau_h=\tau_h-\tau_{h-1}$ is $O(\exp(-ch))$ for some $c>0$, so that $y(\rho,u)$ converges for $u<\exp(c)$.
Hence, by Lemma~\ref{lem:tail}, we conclude that $\xi(T_n)\leq n^{\e}$ a.a.s.\ with exponential rate.
\end{proof}

\begin{lemma}\label{lem:diamtau}
For $x>0$, the block-decomposition tree $\tau$ of a random connected planar graph with $n$ vertices
and weight $x$ at edges has diameter
at most $n^{\e}$ a.a.s.\ with exponential rate.
\end{lemma}
\begin{proof}
Let $C$ be a pointed connected planar graph, and $\tau$ the associated Bv-tree, rooted at the marked vertex of $C$.
Define the block-height $h(\tau)$ of $\tau$ as the maximal number of blocks (B-nodes) over all paths starting from the root.
Clearly $\D(\tau)\leq 4h(\tau)+4$. In addition the block-height is clearly a height-parameter for the equation
$$
y=F(z,y),\ \ \ \mathrm{where}\ F(z,y)=z\exp(B'(y))
$$
satisfied by the (weighted) generating function $y(z)=zC'(z)$ of pointed connected planar graphs. It is shown in~\cite{gimeneznoy}
that $y'(z)$ converges at its radius of convergence $\rho$. Hence the equation is critical; by Lemma~\ref{lem:height2},
$h(\tau)\leq n^{\e}$ a.a.s.\ with exponential rate, hence $\D(\tau)\leq n^{\e}$ a.a.s.\ with exponential rate.
\end{proof}

Lemma~\ref{lem:diamtau} easily implies that the diameter of a random connected planar graph $C$ with $n$
vertices is at most $n^{1/4+\e}$ a.a.s.\ with exponential rate. Indeed, calling
$\tau$ the block-decomposition tree of $C$ and $B_i$ the blocks of $C$, one has
$$
\D(C)\leq (\D(\tau)+1)\cdot\mathrm{max}_i(\D(B_i)).
$$
Lemma~\ref{lem:diamtau} ensures that $\D(\tau)\leq n^{\e}$ a.a.s.\ with exponential rate.
Moreover Theorem~\ref{theo:2c}
easily implies that a random 2-connected planar graph of size $k\leq n$
 has diameter at most $n^{1/4+\e}$ a.a.s.\ with exponential rate, whatever $k\leq n$ is (proof by splitting in two cases: $k\leq n^{1/4}$ and
$n^{1/4}\leq k\leq n$, similarly as in the proof of Theorem~\ref{theo:diam2c}).
Hence, since each of the blocks has size at most $n$,
 $\mathrm{max}_i(\D(B_i))\leq n^{1/4+\e}$ a.a.s.\ with exponential rate. Therefore we have

\begin{theorem}\label{theo:cx}
For $x>0$, the diameter of a random connected planar graph with $n$ vertices and weight
$x$ at edges is, a.a.s.\ with exponential rate, in the interval $    (n^{1/4- \e},
n^{1/4+\e})$.
\end{theorem}

We can now complete the proof of Theorems~\ref{theo:main} and~\ref{theo:main2}.
Theorem~\ref{theo:main} is just Theorem~\ref{theo:cx} for $x=1$.
To show Theorem~\ref{theo:main2}, one uses the fact (proved in~\cite{gimeneznoy}) that for each $\mu\in(1,3)$
there exists $x>0$ such that a random connected
planar graph with $n$ edges and weight $x$
at edges has probability $\Theta(n^{-1/2})$ to have $\lfloor \mu n\rfloor$ edges.\\


\section{Proof of Theorem~\ref{theo:Nm}}\label{sec:proof_theo_delayed}
The proof of Theorem~\ref{theo:Nm} follows the same lines as the proof of Theorem~\ref{theo:cx},
with the RMT-tree playing the role that the Bv-tree had in Theorem~\ref{theo:cx}.
The lower bound is obtained from the fact, established in Lemma~\ref{lem:big3conn},
that a random planar network
has a.a.s.\ a ``big'' 3-connected component. The upper bound is obtained
from the inequality given in Section~\ref{subsec:2to3},
\begin{equation}\label{eq:ineqprooftheo}
\D(G)\leq \mathrm{max}_i(\D(B_i))\cdot(\D(\tau)+1)\cdot  \mathrm{max}_{(u,v)\in\mathcal{E}_{\mathrm{virt}}}\mathrm{Dist}_G(u,v)
\end{equation}
where $G$ is  the 2-connected
planar graph obtained by connecting the two poles of the considered planar network,
$\tau$ is the RMT-tree of $G$, the $B_i$'s are
the bricks of $G$, and $\mathcal{E}_{\mathrm{virt}}$ is the set of virtual edges of $G$.
To get the upper bound we will successively prove that a.a.s.\ with exponential rate we have
 $\D(\tau)\leq m^{\e}$ (in Lemma~\ref{lem:diamtaunetwork}),
 $\mathrm{max}_i(\D(B_i))\leq m^{1/4+\e}$ (in Lemma~\ref{lem:diam_bricks}), and
$\mathrm{max}_{(u,v)\in\mathcal{E}_{\mathrm{virt}}}\mathrm{Dist}_G(u,v)\leq m^{\e}$
(in Lemma~\ref{lem:maxde}).

First we need the following lemma, which is a counterpart of Lemmas~\ref{lem:root-face-degree-planar-maps} and~\ref{lem:root-face-2conn}
for 3-connected maps.
\begin{lemma}\label{lem:root-face-degree-3-conn}
Let $T(z,u)$ be the generating function of  rooted $3$-connected maps
  where $z$ marks the number of non-root
  edges, $u$ marks
the root-face degree, and with weight $x$ at each vertex not incident
to the root-edge. Let $\rho$ be the radius
of convergence of $T(z,1)$.
Then there is $u_0>1$ such that $T(\rho,u_0)$ converges.
In addition for $0<a<b$, the value of $u_0$ can be chosen uniformly over $x\in[a,b]$,
and $M_i(z,u_0)$ is uniformly bounded over $x\in[a,b]$.
\end{lemma}
\begin{proof}
The result is derived from Lemma~\ref{lem:root-face-2conn}
using a bivariate version of Equation~\eqref{eq:planeNetworks},
in the very same way that Lemma~\ref{lem:root-face-2conn} is derived
 from Lemma~\ref{lem:root-face-degree-planar-maps} using a bivariate
 version of
 Equation~\eqref{eq:MCeq}.
\end{proof}

To carry out the proof it is useful to rely on a well-known recursive decomposition of planar networks 
that derives from the RMT-tree.
Call a planar network $D$ \emph{polyhedral} if the poles are not adjacent and the addition of an edge between the poles
gives a $3$-connected planar graph with at least $4$ vertices.
Similarly as in the case of embedded graphs (see Section \ref{sec:3ccore}), a planar network is either obtained as several planar networks in series (S-network),
or as several planar networks in parallel (P-network), or as a polyhedral planar network where each edge is substituted by an arbitrary planar network (H-network). 
This can also be seen using the RMT-tree. Indeed let $B=D+e$ be the $2$-connected planar graph obtained from $D$
by adding an edge $e$ between the two poles, and let $\tau$ be the RMT-tree of $B$. Then $e$ corresponds to a leaf $\ell$ of $\tau$,
and the type of the inner node $\nu$ of $\tau$ adjacent to $\ell$ gives the type of the planar network (S-network  if $\nu$ is an R-node,
P-network if $\nu$ is an M-node, H-network if $\nu$ is a T-node).
Let $D\equiv D(z)$, $S\equiv S(z)$, $P\equiv P(z)$, $H\equiv H(z)$
be respectively the generating functions of planar networks,
series-networks, parallel networks, and polyhedral networks, where $z$ marks the number of edges
and with weight $x$ at each non-pole vertex.
And let $T(z)$ be the series of edge-rooted 3-connected planar graphs
where $z$ marks the number of non-root edges.
One finds (see~\cite{Walsh}):

\begin{equation}\label{eq:syst_networks}
\left\{
\begin{array}{lll}
D&=&z+S+P+H,\\
S&=&(z+P+H)xD,\\
P&=&(1+z)\exp(S+H)-1-z-S-H,\\
H&=&T(D).
\end{array}
\right.
\end{equation}

The equation system above
is similar to the one for plane networks; the difference is that for planar networks assembled in parallel, the order does
not matter (since the graph is not equipped with a plane embedding).
Note that the $2$nd equation gives $S=(D-S)xD$, i.e., $S=xD^2/(1+xD^2)$,
and the $3$rd equation gives $z+S+P+H=(1+z)\exp(S+H)-1$. Since $D=z+S+P+H$,
we finally obtain
\begin{equation}
D=(1+z)\exp\left(\frac{xD^2}{1+xD}+T(D)\right)-1.
\end{equation}

\begin{lemma}\label{lem:big3conn}
For $x>0$, let $N$ be a random planar network with $m$ (labelled) edges and weight $x$ at (unlabelled) vertices.
Then $N$ has a $3$-connected component (a $T$-brick in the tree-decomposition) of size at least $m^{1-\e}$ a.a.s.\ with exponential rate.
\end{lemma}
\begin{proof}
The proof is very similar to the one of Lemma~\ref{lem:bigblockC}. For $k\geq 1$ define $T_k(z)$ as the weighted generating function of rooted $3$-connected planar
graphs with at least $4$ vertices and at most $k$ edges,
where $z$ marks the number of non-root edges, with
weight $x$ at non-pole vertices (hence $T(z)=\lim_{k\to\infty}T_k(z)$).
And define $D_k\equiv D_k(z)$ as the weighted generating function of planar networks with weight $x$ at vertices, and where all $3$-connected components ($T$-bricks)
have at most $k$ edges. Then clearly
$$
D_k=(1+z)\exp\left(\frac{xD_k^2}{1+xD_k}+T_k(D_k)\right)-1,
$$
so $T_k$ and $D_k$ are related by the same equation as $T$ with $D$.
Note that the functional inverse of $D$ is the function $\phi(u)=(u+1)\exp(-xu^2/(1+u)-T(u))-1$
and the functional inverse of $D_k$ is the function $\phi_k(u)=(u+1)\exp(-xu^2/(1+u)-T_k(u))-1$.
The arguments are then the same as in the proof of Lemma~\ref{lem:bigblockC}: one defines $u_k=R(1+1/(k\log(k)))$,
where $R$ is the radius of convergence of $\phi(u)$
(it is proved in~\cite{BeGa} that $R$ is also the radius of convergence
of $T(u)$ and that $a=\phi'(R)$ is strictly positive),
and one proves that for $k$ large enough and $n\geq 0$,
$$
[z^n]D_k\leq 2\left(\rho+\frac{a}{2k\log(k)} \right)^{-n},
$$
where $\rho=\phi(R)$ is the radius of convergence of $D(z)$. One concludes the proof using the
fact, proved in~\cite{BeGa}, that $[z^n]D(z)=\Theta(\rho^{-n}n^{-5/2})$.
\end{proof}

Note that Lemma~\ref{lem:big3conn} directly gives the lower bound in Theorem~\ref{theo:Nm},
using the fact (proved in Theorem~\ref{theo:D3cgraph})
that the diameter of a random 3-connected planar graph
of size $k$ is at least $k^{1/4-\e}$ a.a.s.\ with exponential rate.

The rest of the section is now devoted to the proof of the upper bound  in Theorem~\ref{theo:Nm}.
Let $D$ be a random planar network with $m$ labelled edges and weight $x>0$ at vertices,
let $G$ be the 2-connected planar graph obtained by connecting the poles of $D$, and
let $\tau$ be the RMT-tree of $G$.
To show that $\D(\tau)\leq n^{\e}$
 we need to extend Lemma~\ref{lem:height} to vectorial equation systems.
Assume $\vy\equiv (y_1(z),\ldots,y_r(z))$
 satisfies an equation of the form
\begin{equation}\label{eq:tree_vec}
\vy=\vF(z,\vy),
\end{equation}
with $\vF(z,y)$ an $r$-vector of bivariate functions $F_{i}(z,\vy)$ each with nonnegative coefficients,  analytic around $(0,0)$, with $F_i(0,y)=0$.
Assume also that at least one of the $F_{i}$ is
 nonaffine in one of the $y_j$s, and that the dependency graph for $\vF$ (i.e., there is an edge from $i$ to $j$ if $\partial_i F_{j}\neq 0$)
is strongly connected.
The two latter conditions imply that $\vy(\rho)$ is
finite; let $\vtau=\vy(\rho)$. Define $\Jac \vF(z,\vy)$ as the $r\times r$ matrix $M=(M_{i,j})$ of formal power series in $(z,\vy)$
where $M_{i,j}=\partial_iF_j$.
Equation~\ref{eq:tree_vec} is called \emph{critical} if the largest eigenvalue of
$\Jac(\rho,\vtau)$ (which is a real number by the Perron Frobenius theory) is strictly smaller than $1$,
which is also equivalent to the fact that $\vy'(z)$ converges at $\rho$.

Assume that, for $i\in[1..r]$, $y_i(z)$ is the weighted generating function of a combinatorial class $\cG_i$.
A \emph{height-parameter} for~\eqref{eq:tree_vec} is a parameter $\xi$ for the classes $\cG_i$ such that, if we define
$$
y_{i,h}(z)=\sum_{\alpha\in\cG_i,\ \xi(\alpha)\leq h}w(\alpha)z^{|\tau|},\ \ \ \yhv=(y_{1,h},\ldots,y_{r,h}),
$$
then we have
$$
\vy_{h+1}=\vF(z,\vy_h)\ \ \mathrm{for}\ h\geq 0,\ \ \vy_0=0.
$$

As an easy extension of Lemma~\ref{lem:height2} relying on standard arguments of 
 the Perron-Frobenius theory, one has the following extension of Lemma~\ref{lem:height2}:

\begin{lemma}\label{lem:height3}
Let $\cT$ be a combinatorial class endowed with a weight-function $w(\cdot)$ so that the corresponding (weighted) generating function
$y(z)$ is the first component of a vector $\vy=(y_1(z),\ldots,y_r(z))$ of generating functions
 satisfying an equation~\eqref{eq:tree_vec} that is critical.

Let $\xi$ be a height-parameter for~\eqref{eq:tree_vec} and let $T_n$ be
taken at random in $\cT_n$ under the weighted distribution in size
$n$. Assume that $[z^n]y(z)=\Omega(n^{-\alpha}\rho^{-n})$ for some $\alpha$.
Then $\xi(T_n)\leq n^{\e}$ a.a.s.\ with
exponential rate.
\end{lemma}

\begin{lemma}\label{lem:diamtaunetwork}
For $0<a<b$, the RMT-tree $\tau$ of a random planar network with $m$ (labelled) edges
and weight $x$ at vertices has diameter
at most $m^{\e}$ a.a.s.\ with exponential rate, uniformly over $x\in[a,b]$.
\end{lemma}
\begin{proof}
Let $B$ be an edge-rooted $2$-connected planar graph, and $\tau$ the associated RMT-tree, rooted at the leaf corresponding to the root-edge of $B$.
Define the \emph{brick-height}
$h(\tau)$ of $\tau$ as the maximal number of bricks (nodes of type R, M, or T) over all paths starting from the root.
Clearly $\D(\tau)\leq 2h(\tau)+4$.
In addition the brick-height is clearly a height-parameter for the equation-system
\begin{equation}\label{eq:syst_networks2}
\left\{
\begin{array}{lll}
S&=&\frac{x(z+P+H)^2}{1-x(z+P+H)},\\
P&=&(1+z)\exp(S+H)-1-z-S-H,\\
H&=&T(z+S+P+H).
\end{array}
\right.
\end{equation}
which is equivalent to~\eqref{eq:syst_networks}.
Moreover it follows from the results in~\cite{BeGa}
that~\eqref{eq:syst_networks2} is critical
(e.g. because the derivative of the generating function of planar networks converges
at the dominant singularity).
Hence the brick-height of a random planar network with $m$ labelled edges
and weight $x$ at vertices has diameter at most $m^{\e}$ a.a.s.\ with exponential rate,
and the calculations are readily checked to hold uniformly over $x\in[a,b]$.
\end{proof}

\begin{lemma}\label{lem:diam_bricks}
Let $0<a<b$, and let $x\in[a,b]$.
Let $D$ be a random 2-connected planar graph with $m$ labelled edges and weight $x$ at vertices.
Let $G$ be the 2-connected planar graph obtained by connecting the two poles of $D$, and
let $B_1,\ldots,B_k$ be the bricks of $G$. Then $\mathrm{max}(\D(B_i))\leq n^{1/4+\e}$
a.a.s.\ with exponential rate, uniformly over $x\in[a,b]$.
\end{lemma}
\begin{proof}
Consider a brick $B_i$ of $G$. If $B_i$ is 3-connected and conditioned to have $m_i$ edges, $B_i$
 is a random 3-connected planar graph with $m_i$ edges and weight $x$ at vertices.
Hence, according to Theorem~\ref{theo:D3cgraph}, the diameter of $B_i$ is at most $m^{1/4+\e}$
a.a.s.\ with exponential rate (uniformly over $x\in[a,b]$). Now a brick $B_i$
can also be a multiedge-graph, in which case $\D(B_i)=1$, or can be a ring-graph (polygon)
with diameter $\lfloor m_i/2\rfloor$ (with $m_i$ the number of edges of $B_i$).
So it remains to show that the largest R-brick of $G$ is of size at most $m^{\e}$
a.a.s.\ with exponential rate (uniformly over $x\in[a,b]$). Let $A(z,u)$ be the generating
function of 2-connected planar graphs with a marked oriented R-brick, where
$z$ marks the number of edges, $u$ marks the size of the marked R-brick, and
with weight $x$ at vertices. Clearly $A(z,u)$ is given by
$$
A(z,u)=\log\left(\frac{1}{1-ux(D(z)-S(z))}\right).
$$
Let $\rho$ be the radius of convergence of $D(z)$. Note that $S(z)=x(D(z)-S(z))^2/(1-x(D(z)-S(z))$.
Since $S(z)$ converges at $z=\rho$ (as proved in~\cite{BeGa}), we have $x(D(\rho)-S(\rho))<1$,
so that $A(z,u)$ is finite for $z=\rho$ and $u$ in a neighborhood of $1$.
Hence by Lemma~\ref{lem:tail}, the distribution of the size of the marked R-brick has exponentially fast
decaying tail. This ensures in turn that the largest R-brick is of size at most $m^{\e}$
a.a.s.\ with exponential rate. And the estimates are readily checked to hold
uniformly for $x\in[a,b]$.
\end{proof}

Consider the following parameter $\chi$ defined recursively for each planar network $N$:
\begin{itemize}
\item
If $N$ is reduced to a single edge, then $\chi(N)=1$.
\item
If $N$ is made of several planar networks $N_1,\ldots,N_k$ in parallel or in series, then
$\chi(N)=\chi(N_1)+\cdots+\chi(N_k)$.
\item
If $N$ has a 3-connected core $T$, and if $N_1,\ldots,N_k$ are the planar networks
substituted at the edges of the outer face of $T$, then $\chi(N)=\chi(N_1)+\cdots+\chi(N_k)$.
\end{itemize}
It is easy to check recursively that $\chi(N)$ is at least the distance between the two poles of $N$.
For each $x>0$, denote by $D(z,u)$ (resp. $S(z,u)$, $P(z,u)$, $H(z,u)$)
the bivariate generating function of planar networks (resp. series-networks, parallel networks, polyhedral
networks) where $z$ marks the
number of edges, $u$ marks the parameter $\chi$, and with weight $x$ at each non-pole vertex.
Let $T(z,u)$ be the series of edge-rooted 3-connected planar graphs
where $z$ marks the number of non-root edges and $u$ marks the number
of non-root edges incident to the outer face, and
with weight $x$ at each vertex not incident to the root-edge.
Then (with $D(z)=D(z,1)$):
\begin{equation}\label{eq:syst_biv}
\left\{
\begin{array}{lll}
D(z,u)&=&zu+S(z,u)+P(z,u)+H(z,u),\\
S(z,u)&=&(zu+P(z,u)+H(z,u))xD(z,u),\\
P(z,u)&=&(1+zu)\exp(S(z,u)+H(z,u))-1-zu-S(z,u)-H(z,u),\\
H(z,u)&=&T(D(z),D(z,u)/D(z)).
\end{array}
\right.
\end{equation}
which coincides with~\eqref{eq:syst_networks} for $u=1$.

\begin{lemma}\label{lem:dist_pole_network}
For  each $x>0$, let $\rho$ be the radius of convergence of $D(z,1)$.
Then there exists $u_0>1$ such that the generating function $D(\rho,u_0)$ converges.
 In addition, for $0<a<b$ there exists some value $u_0>1$ and some constant
$C>0$ that works uniformly over $x\in[a,b]$,
and such that $D(\rho,u_0)<C$ for $x\in[a,b]$.
\end{lemma}
\begin{proof}
Let $R=T(D(\rho),1)$. As shown in~\cite{BeGa}, $R$ is the radius of convergence of $w\to T(w,1)$.
In addition, Lemma~\ref{lem:root-face-degree-3-conn}
 ensures that there is some $v_0>1$ such that  $T(R,v_0)$ converges.
It follows from the results in~\cite{BeGa} that, at $z=\rho$  the largest
eigenvalue of the Jacobian matrix of~\eqref{eq:syst_networks2}
is strictly smaller than $1$. Hence by continuity,
at $z=\rho$ the largest eigenvalue of the Jacobian matrix of~\eqref{eq:syst_biv}
is strictly smaller than $1$ in a neighborhood of $u=1$. Hence $D(\rho,u)$ converges for $u$ close to $1$.
Finally, the uniformity of the statement for $x\in[a,b]$ follows from the uniformity over $x\in[a,b]$
in Lemma~\ref{lem:root-face-degree-3-conn} and from the
 fact that~\eqref{eq:syst_biv} is continuous according to~$x$.
\end{proof}

Let $G$ be a 2-connected planar graph with a marked virtual edge $e=\{v,v'\}$.
The edge $e$ corresponds to an edge $e^*$ in the RMT-tree connecting
two nodes $\nu_1$ and $\nu_2$. The subtree of the RMT-tree hanging from $\nu_1$ (resp. $\nu_2$)
corresponds to a planar network $N_1$ (resp. $N_2$).
Define $\wchi(G)=\chi(N_1)+\chi(N_2)$. Clearly $\wchi(G)$ is an upper bound
on the distance (in $G$) between $v$ and $v'$.
We denote by $G(z,u)$ the generating function of 2-connected planar graphs
with a marked virtual edge, where $z$ marks the number of edges and $u$
marks the parameter $\wchi$.
Looking at the possible types for the nodes $\nu_1$ and $\nu_2$,
we obtain (the terms $S(z,u)^2$
and $P(z,u)^2$ do not appear since there are no adjacent R-nodes nor adjacent M-nodes
in the RMT-tree):
$$
G(z,u)=2S(z,u)P(z,u)+2S(z,u)H(z,u)+2P(z,u)H(z,u)+H(z,u)^2.
$$

\begin{lemma}\label{lem:dist_pole_2conn_marked_virtual_edge}
For  each $x>0$, let $\rho$ be the radius of convergence of $G(z,1)$.
Then there exists $u_0>1$ such that the generating function $G(\rho,u_0)$ converges.
 In addition, for $0<a<b$ there is some value $u_0>1$ that works uniformly over $x\in[a,b]$,
and such that $G(\rho,u_0)=O(1)$ for $x\in[a,b]$.
\end{lemma}
\begin{proof}
First the expression of $G(z,u)$ in terms of the generating functions of planar networks ensures
that $\rho$ is the radius of convergence of $D(z,1)$, and that the property for $G(z,u)$
is just inherited from the same property satisfied by $D(z,u)$ (and the other
network generating functions $S(z,u)$, $P(z,u)$, $H(z,u)$) that has
been proved in Lemma~\ref{lem:dist_pole_network}.
\end{proof}

\begin{lemma}\label{lem:maxde}
For $0<a<b$ and $x\in[a,b]$,
let $D$ be a random planar network with $m$ (labelled) edges and weight $x$ at vertices.
Let $G$ be the 2-connected planar graph obtained by connecting the pole of $D$.
For each virtual edge $e=\{u,v\}$ of $G$, let $d_e$ be the distance in $G$ between $u$ and $v$, and let
$\dmax$ be the maximum of $d_e$ over all virtual edges of $G$. Then $\dmax\leq m^{\e}$ a.a.s.\ with exponential rate, uniformly over $x\in[a,b]$.
\end{lemma}
\begin{proof}
A planar network $N$ with a marked virtual edge $e$ can be seen as a 2-connected planar graph $G$
rooted at a virtual edge $e=\{u,v\}$ and with a secondary marked edge whose ends
play the role of poles of the planar network.
Let $G$ be a random 2-connected planar graph rooted at a virtual edge, with $m$ edges
and weight $x$ at vertices.
By Lemma~\ref{lem:dist_pole_2conn_marked_virtual_edge},
the distribution of the distance between $u$ and $v$ in $G$ has exponentially
fast decaying tail. Hence, for $N$ a random planar network with $m$ edges, weight $x$ at vertices,
and with a marked virtual edge $e=\{u,v\}$, the distribution of the distance $d_e$ between $u$ and $v$ in $G$ has exponentially
fast decaying tail as well.
In addition it is easy to prove inductively (on the number of nodes in the RMT-tree)
that a planar network with $m$ edges has $O(m)$ virtual edges.
Hence $\dmax\leq m^{\e}$ a.a.s.\ with exponential rate, and the uniformity over $x\in[a,b]$
follows from the uniformity over $x\in[a,b]$ in Lemma~\ref{lem:dist_pole_2conn_marked_virtual_edge}.
\end{proof}

To conclude, Lemmas~\ref{lem:diamtaunetwork},~\ref{lem:diam_bricks}, and~\ref{lem:maxde} together with the inequality~\eqref{eq:ineqprooftheo} yield the upper bound in Theorem~\ref{sec:proof_theo_delayed}.

\section{Diameter estimates for subcritical graph families}\label{sec:subcritical}

We conclude with a remark on so-called ``subcritical'' graph families,
these are the families where the system
\begin{equation}\label{eq:subc}
y=z\exp(B'(y))
\end{equation}
to specify pointed connected from pointed 2-connected graphs in the family
is admissible, i.e., $F(z,y)=z\exp(B'(y))$ is analytic at $(\rho,\tau)$ where $\rho$ is
the radius of convergence of $y=y(z)$ and $\tau=y(\rho)$.
Examples of such families are cacti graphs, outerplanar graphs, and series-parallel graphs.

Define the \emph{block-distance} of a vertex $v$ in a vertex-pointed connected graph $G$
as the minimal number of blocks one can use to travel from the pointed vertex
to $v$; and define the \emph{block-height} of $G$ as the maximum of the
block-distance over all vertices of $G$. With the terminology of
Lemma~\ref{lem:height},
one easily checks that the block-height is a height-parameter 
for~\eqref{eq:subc}. Hence by Lemma~\ref{lem:height}, the block-height $h$ of a
random pointed connected graph $G$ with $n$ vertices
from a subcritical family is in $(n^{1/2-\e},n^{1/2+\e})$
a.a.s.\ with exponential rate. Clearly $\D(G)\geq h-1$ since the distance between two
vertices is at least the block-distance minus 1. Hence $\D(G)\geq n^{1/2-\e}$
a.a.s.\ with exponential rate.
For the upper bound, note that $\D(G)\leq h\cdot\mathrm{max}_i(|B_i|)$], where
the $B_i$'s are the blocks of $G$. Using Lemma~\ref{lem:tail} and the subcritical condition
one easily shows that $\mathrm{max}_i(|B_i|)\leq n^{\e}$ a.a.s.\  with exponential rate.
This implies that $\D(G)\leq n^{1/2+\e}$ a.a.s.\ with exponential rate.
It would be interesting to obtain explicit limit laws (in the scale $n^{1/2}$) for the
diameter of random graphs in subcritical families such as outerplanar graphs and series-parallel graphs.
Such a result has for instance recently been obtained for stacked triangulations~\cite{AlMa08}.

\vspace{.4cm}

\noindent{\bf Additional note.} After this paper was written and reviewed, Ambj\o rn and Budd~\cite{AmBu13} found an explicit expression for the 2-point function of planar (embedded) maps, that could simplify a bit the content of Section~\ref{sec:maps} by avoiding the detour via quadrangulations. Unfortunately this simplification would not affect the other sections (indeed~\cite{AmBu13} does not apply to 2- or 3- connected maps) and thus it would not enable us to get more precise results than the ones we got here.

\bibliographystyle{plain}
\bibliography{diam}

\end{document}